\documentclass[11pt,letterpaper]{article} 
\usepackage[margin=1in]{geometry}
\usepackage{xcolor}
\usepackage{multicol}
\usepackage{listings}% http://ctan.org/pkg/listings
\usepackage{amsmath, amsthm, amssymb, wasysym, verbatim, bbm, color, graphics}
\usepackage{url}
\usepackage{mathtools}
\usepackage{hyperref}
\usepackage{cleveref}

\usepackage{enumitem} % list in roman
\usepackage{psfrag} %for figures

\usepackage{hyperref}
\hypersetup{colorlinks,linkcolor=blue}
\usepackage{romannum}

\usepackage{academicons}
\definecolor{orcidlogocol}{HTML}{A6CE39}
\usepackage{lineno}
\newcommand{\orcidicon}{
  \includegraphics[height=1.8ex]{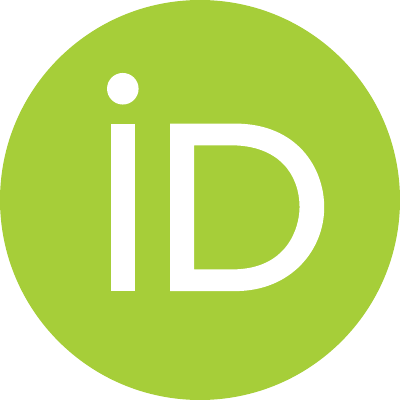}
}

\usepackage{fancyhdr}
\newlength\FHoffset
\fancyheadoffset{\FHoffset}

\newlength\FHleft
\newlength\FHright

\renewcommand{\headrulewidth}{1.0pt} 
\newbox\FHline
\setbox\FHline=\hbox{\hsize=\paperwidth%
  \hspace*{\FHleft}%
  \rule{\dimexpr\headwidth-\FHleft-\FHright\relax}{\headrulewidth}\hspace*{\FHright}%
}

\title{\textbf{
Mixing and Small-Scale Formation in a Passive Divergence-Free Vector Field}}
\author{
\textbf{Anuj Kumar}\thanks{Department of Mathematics, University of California Davis, CA 95616, USA. \textit{Email:}  \href{mailto:anku@ucdavis.edu}{anku@ucdavis.edu}.
\\ \hspace*{0.5cm} \href{https://orcid.org/0000-0002-9203-9177}{\orcidicon orcid.org/0000-0002-9203-9177}
}
\and
\textbf{Franziska Weber}\thanks{Department of Mathematics, University of California Berkeley, CA 94720, USA. \textit{Email:}  \href{mailto:fweber@math.berkeley.edu}{fweber@berkeley.edu}.
\\ \hspace*{0.5cm} \href{https://orcid.org/0000-0002-7159-0170}{\orcidicon orcid.org/0000-0002-7159-0170}\\
F.W.'s work was supported in part by the U.S. Department of Energy, Office of Science, Office of Advanced Scientific Computing Research's Applied Mathematics Competitive Portfolios program under Contract No. AC02-05CH11231 and by National Science Foundation award DMS 2438083.\\
A.K. and F.W. also acknowledge Savio computational cluster resources provided by the Berkeley Research Computing program at the University of California, Berkeley (supported by the UC Berkeley Chancellor, Vice Chancellor for Research, and Chief Information Officer).
}
}

\date{}

\usepackage{xfrac}
\usepackage{float}
\usepackage[title]{appendix}
\usepackage{todonotes}
\usepackage{subcaption}

\newcommand{\R}{\mathbb{R}}

\newcommand{\Z}{\mathbb{Z}}
\newcommand{\N}{\mathbb{N}}

\newcommand{\T}{\mathbb{T}}

\newcommand{\p}{\mathbf{P}}

\newcommand\norm[1]{\left\lVert#1\right\rVert}
\newcommand{\Grad}{\nabla}
\newcommand{\Div}{\operatorname{div}}
\newcommand{\dom}{\T^d}

\newtheorem{lemma}{Lemma}[section]

\newtheorem{theorem}[lemma]{Theorem}
\newtheorem{remark}[lemma]{Remark}

\newtheorem*{maintheorem*}{Main Theorem}
\theoremstyle{definition}{}

\theoremstyle{note}{
\newtheorem*{claim*}{Claim}}

\allowdisplaybreaks

\numberwithin{equation}{section}

\begin{document}
 \pagenumbering{arabic}
 \maketitle
\begin{abstract}
We study mixing for a divergence-free passive vector field $u$ transported by another divergence-free vector field $U$, where $u$ evolves according to
$ \partial_t u + (U \cdot \nabla) u + \nabla p = 0.$
In recent years, a lot of attention has been given to the question of optimal mixing in the scalar case, where there is a Sobolev constraint on the advecting velocity. In the vector setting considered here, however, the pressure term introduces substantial difficulties, since the simple Lagrangian perspective available in the scalar case is no longer applicable.

In this paper, we investigate mixing on a torus $\mathbb{T}^d$ under the assumption that the field $U$ satisfies
$ \|U(t)\|_{W^{1,q}} \leq C $
and we quantify mixing through the decay of the homogeneous $ H^{-\alpha}$ norm of $u$. We start with establishing conditions on $U$ that guarantee existence and uniqueness of solutions. We then derive lower bounds on the mixing rate for various ranges of $q$ and $\alpha$. In addition, we carry out numerical simulations of mixing by choosing, at each time instant, a field $U$ that maximizes the instantaneous decay of the $ H^{-\alpha}$ norm. These simulations provide evidence that the optimal mixing rate is at least exponential in time. More broadly, we view the present model and its diffusive analogue, as a useful framework for probing mechanisms of small-scale formation in divergence-free vector fields and for formulating simplified versions of open questions related to the incompressible Euler and Navier--Stokes equations.
\end{abstract}

\section{Introduction}

Mixing plays a crucial role both in engineering and in natural fluid flows. In chemical engineering, for example, mixing facilitates the desired reaction rates by ensuring thorough interaction of substances \cite{ottino1994mixing}. It also prevents stratification and leads to homogeneous fluid properties, which is crucial in applications like heat exchangers and chemical reactors \cite{paul2004handbook}. In the oceans, mixing drives the meridional overturning circulation, which in turn, is responsible for the state of climate \cite{munk1998abyssal, wunsch2004vertical}. Furthermore, mixing in oceans is vital for the upward transport of nutrients from the seafloor to the surface, an essential component for plankton growth \cite{caldwell1995turbulence}.

Motivated by the considerations above, a natural question from a mathematical perspective is how fast a scalar can be mixed for a given budget of the incompressible advecting vector field $U$. In the past two decades, this question has gained a lot of attention, particularly in the case where the advecting vector field $U$ belongs to the Sobolev space $W^{1, q}$ for $q \in [1, \infty]$. Lin, Thiffeault and Doering investigated this question using numerical methods \cite{Lin2011}. They demonstrated an exponential mixing rate when the velocity field is constrained in $W^{1, 2}$. They quantified mixing through the decay of the $H^{-1}$-norm of the scalar. Crippa and De Lellis \cite{crippa2008regularity}, using a Lagrangian approach, rigorously established an exponentially decaying lower bound for the mixing rate (see also Seis \cite{seis2013maximal} and Iyer, Kiselev \& Xu \cite{iyer2014lower}) when $U \in W^{1, q}$ for $q > 1$, making an important progress on Bressan’s conjecture \cite{bressan2003lemma}.  Subsequent work by Alberti, Crippa \& Mazzucato \cite{AlbertiCrippaMazzucato19}, Yao \& Zlatos \cite{YaoZlatos17} and Elgindi \& Zlatos \cite{elgindiuniversalmixer} provided constructions of the advecting vector field $W^{1, q}$ for which the scalar exhibits an exponential mixing rate, demonstrating the sharpness of the lower bound established by Crippa and De Lellis.

While small scales in the velocity field are introduced in the scalar setting to produce mixing, a natural next interesting question is how such small scales arise in the velocity field in the first place. From a physical perspective, mixing, or more generally, the generation of small scales in the velocity field, is essential for turbulent flows to exhibit enhanced or anomalous dissipation \cite{sreenivasan1984scaling, frisch1995turbulence}. Although our ultimate goal is to understand the formation of small scales at the level of the incompressible Euler equations, in this paper we restrict our attention to the following model, for several reasons explained below.

The primary purpose of this paper is to investigate mixing of an incompressible vector field $u:\R^+\times\T^d\to \R^d$ governed by the system of equations
\begin{subequations}\label{eq:vectortransport}
\begin{align}\label{eq:uevolution}
    \partial_t u + (U\cdot \Grad) u +\Grad p& = 0,\quad (t,x)\in \R^+\times\T^d,\\
    \label{eq:udivconstraint}
    \Div u & = 0,\quad (t,x)\in \R^+\times\T^d,\\
    u(0,x) &  = u_0(x),\quad x\in \T^d\label{eq:initcond}
\end{align}
\end{subequations}
where $U:\R^+\times\T^d\to \R^d$ is a smooth divergence-free advecting vector field and, for some $q \in [1, \infty]$, satisfies
\begin{equation}
    \label{eq:Ubound}
    \norm{U(t)}_{W^{1,q}(\dom)}\leq C,
\end{equation}
for all $t\in\R^+$ and some constant $C<\infty$, and a suitable mean-zero initial condition $u(0,x)=u_0(x)\in L^2(\dom)$. Here, $\T^d$ is the $d$-dimensional torus of side length $L$.

Next, we explain the general reasons for working with this particular model. We also highlight several useful features, viewpoints and challenges associated with it that ultimately motivated our choice.

\begin{enumerate}[label=(\roman*)]
\item The main reason we study vector mixing in this model, rather than for example the linearized Euler equations around $U$ or the vorticity equation, is that model \eqref{eq:vectortransport} has two important structural properties: (a) it admits a natural energy ($L^2$) balance, and (b) when $u=U$, it reduces to the incompressible Euler equations. Since our ultimate goal is to understand the formation of small scales in Euler (and eventually Navier--Stokes), this provides a natural route: one may first analyze the more flexible model \eqref{eq:vectortransport} and then, in future work, try to close the gap between $U$ and $u$.

\item We view model \eqref{eq:vectortransport} as a more flexible version of the incompressible Euler equations (or of the incompressible Navier--Stokes equations, if one adds the term $\nu \Delta u$ to the right-hand side of \eqref{eq:udivconstraint}). In the fluid equations, one requires $u$ and $U$ to coincide pointwise. In contrast, in \eqref{eq:vectortransport} one may impose only a weaker constraint that $U$ and $u$ are ``structurally similar". For instance, one could impose a constraint of the form $U = \mathbb{P}(u \circ \chi),$
for some $C^1$-diffeomorphism $\chi$, where $\mathbb{P}$ denotes the Leray projector. In this preliminary study, however, we take a more permissive viewpoint that $U$ and $u$ are completely independent.

\item This added flexibility allows one to build and test various plausible mechanisms for generating small scales. For example, one may construct a candidate mechanism in $U$, and then investigate whether these mechanisms induce small-scale formation in $u$. If a proposed mechanism already fails at the level of model \eqref{eq:vectortransport}, then it is unlikely to hold for the Euler equations. In this way, the model serves as a useful testing ground for identifying which mechanisms may truly be relevant.

\item Mathematically, the creation of small scales in equations \eqref{eq:vectortransport} is considerably more challenging than in the scalar transport equation. In the transport equation, the scalar $\theta$ follows the Lagrangian trajectories of the velocity field $U$, and mixing can be induced by generating small scales in $U$. This idea has been used extensively in the study of anomalous dissipation in passive scalars \cite{drivas22anomdissp, elgindi2023norm}, optimal mixing of scalars \cite{YaoZlatos17, AlbertiCrippaMazzucato19}, optimal scalar transport \cite{ doering2019optimal, kumar2022three, tobasco2021optimal}, and even problems related to the nonuniqueness of solutions to the transport equation \cite{depauw2003non, kumar2023nonuniqueness, brue2024sharp}. For the vector equation considered here, however, such an approach is no longer available. The pressure term introduces a substantial new difficulty and renders the Lagrangian viewpoint essentially inapplicable. At present, it is not even clear whether the pressure enhances or suppresses the formation of small scales. The framework studied in this paper provides a way to isolate and investigate the role of pressure, without the additional complications introduced by nonlinear terms.

\item Although in this paper we restrict our attention to mixing, we believe that model \eqref{eq:vectortransport} may serve a broader purpose, namely, as a framework for probing conjectures related to the incompressible Euler and Navier--Stokes equations, for instance those concerning anomalous dissipation or singularity formation. Broadly speaking, one may ask the following question for model \eqref{eq:vectortransport} (and, respectively, for the model with the added diffusion term $\nu\Delta u$): if one endows $U$ with certain conjectured properties, do corresponding properties emerge in $u$? If this question cannot be answered even at the level of \eqref{eq:vectortransport} (or its diffusive analogue), then resolving the original conjecture for Euler (or Navier--Stokes) is likely to be substantially more difficult. We outline two such questions of this type in the concluding section.
\end{enumerate}

Once the model has been fixed, we now describe the main objectives of the paper and the questions we address.

Our goal is to investigate the maximal rate of mixing of the vector field \(u\) under a prescribed Sobolev \(W^{1,q}\) budget on the advecting velocity field \(U\). As in the scalar setting \cite{Lin2011, thiffeault2012using, iyer2014lower, YaoZlatos17, AlbertiCrippaMazzucato19}, we quantify mixing by means of the homogeneous norm \({H}^{-\alpha}\), for \(1/2 \leq \alpha\leq 1\), namely \(\|u(t,\cdot)\|_{H^{-\alpha}}\). More precisely, for a given \(W^{1,q}\) budget on \(U\), our goal is to study
\[
\inf_{\substack{\|U\|_{W^{1,q}} \leq 1 \\ \Div U = 0}}
\|u(t,\cdot)\|_{{H}^{-\alpha}}
\]
as a function of time. In this framework, decay of the \({H}^{-\alpha}\) norm of \(u\) is interpreted as a measure of the rate at which \(u\) generates small scales.

In \Cref{sec: existence}, we establish conditions on $U$ that ensure existence, uniqueness, and stability of solutions $u$ to \eqref{eq:vectortransport}. In \Cref{sec: Mixing estimates}, we study the mixing rate and, in particular, derive lower bounds for mixing for various ranges of $q$ and $\alpha$. In \Cref{sec: Numerical algorithm}, we focus on the case of a $W^{1,2}$ budget and perform numerical simulations of optimal mixing by identifying vector fields $U$ that instantaneously minimize the $\| \cdot \|_{{H}^{-\alpha}}$ norm of $u$. Finally, in \Cref{sec: Discussion}, we conclude with a discussion of some directions for future work.

\section{Existence of solutions}
\label{sec: existence}
We start by showing that solutions exist to system~\eqref{eq:vectortransport}. We do so via an approximation by a standard spectral method that will also be used later on to carry out numerical simulations. To define the method, we consider the Fourier transform for $f(t)\in L^2(\T^d)$, 
\begin{equation}
	\label{eq:fourier}
	f(t, x)=\sum_{k \in \mathbb{Z}} \hat{f}_k(t)e^{i2\pi \frac{k\cdot x}{L}},\quad \hat{f}_k(t) = \frac{1}{ L^d}\int_{\T^d} f(t,x)e^{-i2\pi \frac{k\cdot x}{L}} \, {\rm d}x.
\end{equation}

Given a velocity field $U$, we define its
truncation up to $N\in \N$ modes as:
\begin{equation}
	\label{eq:Utrunc}
	P_N U(t,x):= U_N(t,x) = \sum_{|k|_\infty\leq N} \hat{U}_k(t) e^{i2\pi \frac{k\cdot x}{L}},
\end{equation}
where $|k|_\infty = \max_{j=1,\dots, d}|k_j|$. Then, we can define the following spectral scheme: 
For $N\in \N$, we let $u_N(t,x)=\sum_{|k|_\infty\leq N}\hat{u}_k^N(t)e^{i2\pi \frac{k\cdot x}{L}}$, where the $\hat{u}_k^N$ are the solution of the following Cauchy problem:
\begin{subequations}
	\label{eq:numscheme}
	\begin{align}
	\partial_t u_N +\p_N((U_N\cdot \Grad)u_N) &= 0,\\
	u_N(0)= \p_N u_0,\label{eq:initapprox}
	\end{align}
\end{subequations}
here $\p_N$ is the projection onto divergence free functions and Fourier modes with frequency less than $N$, i.e.,
\begin{equation}
	\label{eq:fourierprojection}
	\p_N f(t, x) = \sum_{|k|_{\infty}\leq N} \left(\hat{f}_k(t) -\frac{\hat{f}_k(t)\cdot k}{|k|^2}k\right) e^{i2\pi \frac{k\cdot x}{L}},\quad \hat{f}_k(t) = \frac{1}{ L^d} \int_{\T^d} f(t,x)e^{-i2\pi \frac{k\cdot x}{L}} dx.
\end{equation}
We note that $\p_N$ is self-adjoint,
\begin{equation*}
	\int_{\T^d}\p_N f\cdot g dx = \int_{\T^d} f\cdot \p_N g dx,
\end{equation*}
and satisfies $\p_N^2 = \p_N$.
We also note that since
\begin{equation*}
	\Div v = \frac{i 2\pi}{L}\sum_k (\hat{v}_k\cdot k)e^{i2\pi \frac{k\cdot x}{L}}, 
\end{equation*}
we have that $v$ is divergence free if and only if its Fourier coefficients satisfy $k\perp \hat{v}_k$ for all $k$. 
We also note that the projection onto the first $N$ Fourier modes satisfies for sufficiently smooth functions, $s\geq 0$, 
\begin{equation}
	\norm{(I-P_N) f}_{H^s}\leq C_s \psi(N)\stackrel{N\to\infty}{\longrightarrow} 0.
\end{equation}
Here we denoted by $H^s$ for $s\in\R$ the homogeneous fractional Sobolev space norm (with the convention that $H^0:=L^2$):
\begin{equation*}
    \norm{f}_{H^s(\dom)}: = \left(\frac{2\pi}{L}\right)^\alpha\norm{\sum_{k\in \Z^d}|k|^\alpha \hat{f}_k e^{i\frac{2\pi}{L}k\cdot x}}_{L^2(\dom)}.
\end{equation*}
Now we can show the following:
\begin{lemma}[Existence of solutions]
	\label{lem:existence}
	A weak solution $u$ of the equations~\eqref{eq:vectortransport} exists for any initial data $u_0\in L^2_{\Div}(\dom)$, in the sense that $u\in L^\infty(\R^+;L^2_{\Div}(\dom))$ satisfies
	\begin{equation}\label{eq:weakform}
		\int_{\R^+}\int_{\dom} \left(u\cdot \partial_t \varphi + [(U\cdot\Grad) \varphi ]\cdot u\right) dx dt+\int_{\dom} u_0\cdot \varphi(0,x) dx =0,
	\end{equation}
	for any smooth and compactly in time supported test function $\varphi \in C_c^\infty(\R^+\times\dom)$.
	In addition, the time derivative $\partial_t u$ satisfies
	\begin{equation*}
		\partial_t u \in \begin{cases}
			L^\infty(\R^+;H^{-1}(\dom)),\quad 	q>d,\\
			L^\infty(\R^+;W^{-1,\frac{2dq}{dq+2d-2q}}(\dom)),\quad q<d.
		\end{cases}
	\end{equation*}
\end{lemma}
\begin{remark}
	Using the integrability of $u$ and $U$, one can use the density of smooth functions to show that~\eqref{eq:weakform} is satisfied for any test function in $L^1(\R^+;H^1(\dom))$ when $q>d$ and $L^1(\R^+;W^{1,\frac{2dq}{dq-2d+2q}}(\dom))$ when $q<d$.
\end{remark}
\begin{proof}
Scheme~\eqref{eq:numscheme} is a system of coupled ODEs for the $\hat{u}^N_k(t)$, $|k|_{\infty}\leq  N$ and since the right hand side expressions are polynomials of $\hat{u}_k^N$, they are locally Lipschitz, and short time existence of a unique solution follows. In order to get existence for arbitrary times, we show a uniform in time and $N$ $L^2$-bound on $u_N(t)$: We multiply the scheme by $u_N$ and integrate over the domain:
\begin{align*}
	\int_{\dom} u_N\cdot \partial_t u_N dx &= -\int_{\dom} u_N\cdot \p_N((U_N\cdot\Grad)u_N) dx\\
	& = - \int_{\dom} \p_N u_N\cdot ((U_N\cdot\Grad)u_N) dx\\
	& = -\int_{\dom} u_N\cdot ((U_N\cdot\Grad)u_N) dx\\
	& = 0.
\end{align*}
Thus 
\begin{equation*}
	\frac12\frac{d}{dt}\norm{u_N(t)}_{L^2}^2= \int_{\dom} u_N\cdot \partial_t u_N dx = 0,
\end{equation*}
and hence the $L^2$-norm of $u_N$ satisfies for any $t\in (0,\infty)$
\begin{equation*}
	\norm{u_N(t)}_{L^2}^2 = \norm{u_N(0)}_{L^2}^2 \leq \norm{u_0}_{L^2}^2.
\end{equation*}
Since $\norm{u_N(t)}_{L^2}^2 =\sum_{|k|_\infty\leq N} |\hat{u}_k^N(t)|^2$, this implies a uniform in time bound on the $\hat{u}_k^N$ and hence we can extend the solution of the ODE to all times. Since this estimate is uniform in $N$, we can also extract a weak* convergent subsequence in $L^\infty_{loc}(\R^+;L^2(\dom))$ that converges to $u\in L^\infty(\R^+;L^2_{\Div})$ which is a weak solution of~\eqref{eq:vectortransport} thanks to the linearity of the equation. That the limit satisfies the divergence constraint almost everywhere follows since it is satisfied for each of the approximations and due to the linearity of the constraint. Due to the weak lower semicontinuity of the limit, we obtain
\begin{equation*}
	\norm{u(t)}_{L^2}\leq \liminf_{N\to \infty}\norm{u_N(t)}_{L^2} \leq \norm{u_0}_{L^2}, \quad \textrm{a.e. }  \, t\in \R^+.
\end{equation*}
 To show the estimate on the time derivative, we consider once more the numerical scheme and multiply it by a suitable test function. When $q>d$, we multiply by a test function

 $\varphi\in L^1(0,T; H^1(\dom))$ and integrate:
 \begin{align*}
 	\left|\int_0^T\int_{\dom}\partial_t u_N \cdot \varphi dxdt\right| & = \left|\int_0^T\int_{\dom}\p_N \varphi \cdot (U_N\cdot \Grad )u_N dx dt\right|\\
 	&  = \left|\int_0^T\int_{\dom}u_N \cdot (U_N\cdot \Grad )\p_N \varphi dx dt \right|\\
 	& \leq \norm{u_N}_{L^\infty_t L^2_x} \norm{U_N}_{L^\infty_{t,x}}\norm{\Grad \p_N \varphi}_{L^1_tL^2_x}\\
 	& \leq C\norm{u_N}_{L^\infty_t L^2_x} \norm{U_N}_{L^\infty_tW^{1,q}_x}\norm{\Grad \varphi}_{L^1_t L^2_x}\\
 	&\leq C \norm{\Grad \varphi}_{L^1_t L^2_x},
 \end{align*}
 thus $\partial_t u_N \in L^\infty (0,T;H^{-1}(\dom))$ uniformly in $N$. If $q<d$, we instead estimate 
 \begin{align*}
 	\left|\int_0^T\int_{\dom}\partial_t u_N \cdot \varphi dxdt\right| & = \left|\int_0^T\int_{\dom}\p_N \varphi \cdot (U_N\cdot \Grad )u_N dx dt\right|\\
 	&  = \left|\int_0^T\int_{\dom}u_N \cdot (U_N\cdot \Grad )\p_N \varphi dx dt \right|\\
 	& \leq \norm{u_N}_{L^\infty_t L^2_x} \norm{U_N}_{L^{q^*}_{t,x}}\norm{\Grad \p_N \varphi}_{L^1_tL^{\frac{2q^*}{q^*-2}}_x}\\
 	& \leq C\norm{u_N}_{L^\infty_t L^2_x} \norm{U_N}_{L^\infty_tW^{1,q}_x}\norm{\Grad \varphi}_{L^1_t L^{\frac{2q^*}{q^*-2}}_x}\\
 	&\leq C \norm{\Grad \varphi}_{L^1_t L^{\frac{2q^*}{q^*-2}}_x},
 \end{align*}
 where $q^* = \frac{dq}{d-q}$ is the Sobolev conjugate of $q$.  Hence in this case $\partial_t u_N \in L^\infty(0,T; W^{-1,\frac{2dq}{dq+2d-2q}}(\dom))$ uniformly in $N$ and thus the same holds for the limit.
\end{proof}
However, we do not know whether these weak solutions are unique or satisfy an energy equality instead of the inequality above. If $U$ has more regularity, we can deduce higher order regularity for $u$ also and show that the solution of~\eqref{eq:vectortransport} is unique.
We start by deriving a uniform $H^1$-norm on the approximations produced by~\eqref{eq:numscheme}:

\begin{lemma}
	\label{lem:h1bound}
	If $U\in L^\infty(\R^+;W^{1,q}(\dom))\cap L^1(\R^+;W^{1,\infty}(\dom))$ and is divergence free, and $u_0\in H^1_{\Div}(\dom)$, then the limit $u_N$ obtained through the spectral scheme~\eqref{eq:numscheme} satisfies for almost every $t\in \R^+$,
	\begin{equation}
		\label{eq:h1bound}
		\norm{\Grad u(t)}_{L^2}\leq \norm{\Grad u_0}_{L^2}\exp\left(\int_0^t \norm{\Grad U(s)}_{L^\infty(\dom)}ds\right).
	\end{equation}
    Moreover, $\partial_t u\in L^1(\R^+;L^2(\dom))$.
\end{lemma}
\begin{proof}
	We derive the bound at the level of the approximations $u_N$, and can then, since it is uniform in $N$, pass to the limit $N\to\infty$ to obtain a bound for the limit.
	To derive an estimate on the gradient of $u_N$, we multiply the scheme~\eqref{eq:numscheme} by $-\Delta u_N$, integrate, and integrate by parts to obtain:
	\begin{align*}
		\frac12\frac{d}{dt}\norm{\Grad u_N}_{L^2}^2& = -\int_{\dom}\Delta u_N\partial_t u_N dx\\
		& = \int_{\dom} \Delta u_N \cdot \p_N ((U_N\cdot \Grad)u_N) dx \\
		& = \int_{\dom} \Delta \p_N u_N \cdot  ((U_N\cdot \Grad)u_N) dx \\
		& = \int_{\dom} \Delta u_N \cdot  ((U_N\cdot \Grad)u_N) dx \\
		& = - \int_{\dom} \Grad u_N : \Grad ((U_N\cdot \Grad)u_N) dx \\
		& = - \int_{\dom} \Grad u_N : ((\Grad U_N\cdot \Grad)u_N) dx - \int_{\dom} \Grad u_N : (( U_N\cdot \Grad)\Grad u_N) dx \\
		& = - \int_{\dom} \Grad u_N : ((\Grad U_N\cdot \Grad)u_N) dx  \\
		& \leq \norm{\Grad u_N}_{L^2}^2 \norm{\Grad U_N}_{L^\infty}.
	\end{align*}
	Using Gr\"onwall's inequality and $\norm{\Grad u_N(0)}_{L^2}\leq \norm{\Grad u_0}_{L^2}$, we obtain
	\begin{equation*}
		\norm{\Grad u_N(t)}_{L^2}^2\leq \norm{\Grad u_0}_{L^2}^2 \exp\left(2\int_0^t \norm{\Grad U_N(s)}_{L^\infty}ds\right),
	\end{equation*}
	thus if
	\begin{equation*}
		\int_0^\infty\norm{\Grad U(s)}_{L^\infty} ds \leq C,\quad\text{and}\quad \norm{u_0}_{H^1}\leq C,
	\end{equation*}
	we obtain that $u_N\in L^\infty(0,\infty;H^1(\dom))$ uniformly in $N$. Using the $H^1$-regularity of $u_N$ and the Lipschitz continuity of  $U$, we can also deduce a better estimate on the time derivative by squaring the equation~\eqref{eq:numscheme} and integrating over space and then estimating:
	\begin{equation*}
		\norm{\partial_t u_N(t)} _{L^2}   \leq \norm{U_N(t)}_{L^\infty}\norm{\Grad u_N(t)}_{L^2}\leq C,
	\end{equation*}
	thus in this case $\partial_t u_N\in L^1(\R^+;L^2(\dom))$, uniformly in $N$.
\end{proof}
 In a similar way, we can derive an estimate on the $H^2$-norm of the limit $u$:
 \begin{lemma}
 	\label{lem:h2estimate}
 		If $U\in L^\infty(\R^+;W^{1,q}(\dom))\cap L^1(\R^+;W^{1,\infty}(\dom))\cap L^1(\R^+;W^{2,3}(\dom))$ and is divergence free, and $u_0\in H^2(\dom)\cap H^1_{\Div}(\dom)$, then the limit $u_N$ obtained through the spectral scheme~\eqref{eq:numscheme} satisfies for almost every $t\in \R^+$,
 	\begin{equation}
 		\label{eq:h2bound}
 		\norm{\Grad^2 u(t)}_{L^2}\leq \norm{\Grad^2 u_0}_{L^2}\exp\left(\int_0^t ( \norm{\Delta U(s)}_{L^3}+ \norm{\Grad U(s)}_{L^\infty} ) ds\right).
 	\end{equation}

 \end{lemma}
 
 \begin{proof}
 	Again, we do all the estimates at the level of the approximation $u_N$ and then since they are uniform with respect to $N$, we can pass $N\to\infty$ to obtain them for the limit $u$.
 	First we note that, since we are using periodic boundary conditions, we have
 	\begin{equation*}
 		\norm{\Delta f}_{L^2}^2 = \norm{\Grad^2 f}_{L^2}^2,
 	\end{equation*}
 	using integration by parts. Then we multiply the numerical scheme~\eqref{eq:numscheme} by $\Delta^2 u_N$, integrate over the domain and integrate by parts:
 	\begin{align*}
 		\frac12\frac{d}{dt}\norm{\Delta u_N}_{L^2}^2& = \int_{\dom}\Delta^2 u_N\partial_t u_N dx\\
 		& = \int_{\dom} \Delta^2 u_N \cdot \p_N ((U_N\cdot \Grad)u_N) dx \\
 		& = \int_{\dom} \Delta^2 u_N \cdot  ((U_N\cdot \Grad)u_N) dx \\
 		& = \int_{\dom} \Delta u_N \cdot\Delta ((U_N\cdot \Grad)u_N) dx \\
 		& = \int_{\dom} \Delta u_N\cdot \Div((\Grad U_N\cdot \Grad)u_N +(U_N\cdot \Grad)\Grad u_N) dx \\
 		& = \int_{\dom} \Delta u_N\cdot ((\Delta U_N\cdot \Grad)u_N +2(\Grad U_N\cdot \Grad)\Grad u_N+(U_N\cdot \Grad)\Delta u_N) dx \\
 		& = \int_{\dom} \Delta u_N\cdot ((\Delta U_N\cdot \Grad)u_N +2(\Grad U_N\cdot \Grad)\Grad u_N) dx \\
 		&\leq \norm{\Delta u_N}_{L^2} \norm{\Grad u_N}_{L^6}\norm{\Delta U_N}_{L^3} + \norm{\Delta u_N}_{L^2}\norm{\Grad^2 u_N}_{L^2}\norm{\Grad U_N}_{L^\infty}\\
 		& \leq C( \norm{\Delta U_N}_{L^3}+ \norm{\Grad U_N}_{L^\infty} )\norm{\Delta u_N}_{L^2}^2.
 	\end{align*}
 	Again, with Gr\"onwall's inequality, we obtain
 	\begin{equation*}
 		\norm{\Delta u_N(t)}_{L^2}^2\leq C \norm{\Delta u_N(0)}_{L^2}^2 \exp\left(\int_0^t ( \norm{\Delta U_N(s)}_{L^3}+ \norm{\Grad U_N(s)}_{L^\infty} ) ds\right),
 	\end{equation*}
 	thus if $U$ is bounded in $L^1(0,T;W^{1,\infty}(\dom)\cap W^{2,3}(\dom))$, and the initial data is in $H^2$, then we obtain that $u_N$ is uniformly contained in $L^\infty(0,T;H^2(\dom))$.
 \end{proof}

When $U$ is Lipschitz continuous in space and therefore $u$ is $H^1$-regular in space, we can show that the solution is unique and that the spectral approximations $u_N$ converge strongly to $u$ with a convergence rate depending on the approximation quality of the Fourier truncations. 
Let us start by showing uniqueness, which also implies stability.
\begin{lemma}
	\label{lem:stability}
	If $U\in L^\infty(\R^+;W^{1,q}(\dom))\cap L^1(\R^+;W^{1,\infty}(\dom))$ for some $q\in (1,\infty]$, then weak solutions of~\eqref{eq:vectortransport} are unique and satisfy the stability estimate: For weak solutions $u$ and $v$ with initial data $u_0$ and $v_0$ respectively, we have
	\begin{equation*}
		\int_{\dom} |u(t,x)-v(t,x)|^2 dx \leq \int_{\dom} |u_0-v_0|^2 dx.
	\end{equation*}
\end{lemma}
\begin{proof}
	Assume first that $u_0,v_0\in H^1_{\Div}(\dom)$, and hence the corresponding weak solutions also satisfy $u,v\in L^\infty(\R^+;H^1_{\Div}(\dom))$.  Subtracting the weak formulations for the two solutions from each other, we have
	\begin{equation*}
			\int_0^\infty\int_{\dom} (u-v)\cdot \partial_\tau\varphi  dx d\tau +\int_{\dom} (u_0-v_0)\cdot \varphi(0,x) dx =-\int_0^\infty\int_{\dom} (U\cdot \Grad) \varphi\cdot (u-v) dx d\tau.
		\end{equation*}
		For a given $t>0$, we let $\theta_{\epsilon}(\tau)=\mathbf{1}_{(r,t)}\star \omega_{\epsilon}(\tau)$, where $\omega_\epsilon(s)=\frac{1}{\epsilon}\omega(s/\epsilon)$ and $\omega$ is a mollifier (nonnegative, symmetric and compactly supported with $\int_\R\omega (s)ds =1$) and $r>\epsilon>0$ is a small number. 
		Then we take $\varphi=(u-v)\theta_\epsilon$ as a test function, which is admissible since it is divergence free and in $L^\infty(0,\infty;H^1(\dom))\cap W^{1,\infty}(0,\infty;L^2(\dom))$ and compactly supported in time. Then we get
			\begin{multline*}
			\int_0^\infty\int_{\dom} \left(|u-v|^2\cdot \partial_\tau\theta_{\epsilon} +\frac12\partial_t |u-v|^2 \theta_\epsilon \right)dx d\tau +\int_{\dom} |u_0-v_0|^2\theta_\epsilon(0) dx\\
			 =-\int_0^\infty\int_{\dom} (U\cdot \Grad) (u-v)\cdot (u-v)\theta_\epsilon(t) dx d\tau.
		\end{multline*}
		Note that since $r>\epsilon$, we have $\theta_\epsilon(0)=0$.
		Integrating by parts and using that $U$ is divergence free and $\theta_\epsilon$ is compactly supported, we obtain
		\begin{equation}\label{eq:whatever}
			\frac12\int_0^\infty\int_{\dom} |u-v|^2\cdot \partial_\tau\theta_{\epsilon} dx d\tau
			=0.
		\end{equation}

We have
		\begin{equation*}
			\partial_\tau \theta_\epsilon(\tau) = \int_{\R} \mathbf{1}_{(r,t)}(s)\partial_\tau\omega_\epsilon(\tau - s) ds =-\int_r^t \partial_s\omega_\epsilon(\tau - s) ds = \omega_\epsilon(\tau-r)-\omega_\epsilon(\tau - t).
		\end{equation*}
		Using this,~\eqref{eq:whatever} becomes
			\begin{equation*}
			\frac12\int_0^\infty\int_{\dom} |u-v|^2 (\omega_\epsilon(\tau - t)-\omega_\epsilon(r-\tau))dx d\tau = 0.
		\end{equation*}
		Letting $\epsilon\to 0$ and using the time continuity of $u$ and $v$, we obtain
		\begin{equation*}
			\frac12\int_{\dom} |u(t)-v(t)|^2 dx = \frac12\int_{\dom} |u(r)-v(r)|^2 dx.
		\end{equation*}
		Since we could choose $r$ arbitrarily small, we can use time continuity of $u$ and $v$ to conclude that
		\begin{equation}\label{eq:gaga}
			\int_{\dom} |u(t)-v(t)|^2 dx  = \int_{\dom} |u_0-v_0|^2 dx.
	\end{equation}
    
    For general $u_0,v_0\in L^2_{\Div}(\dom)$, we approximate them by smooth functions (e.g. by mollifying them), thus producing sufficiently regular $u$ and $v$ by the previous Lemma~\ref{lem:h1bound}. Then we note that \eqref{eq:gaga} only depends on the $L^2$-norms of $u(t)$, $v(t)$, $u_0$, $v_0$, thus we can pass to the limit in the approximation and obtain the result for the limiting $u_0,v_0\in L^2_{\Div}(\dom)$. The uniqueness follows by setting $u_0=v_0$.
\end{proof}
Next, we prove a convergence rate of the spectral scheme~\eqref{eq:numscheme}.
\begin{lemma}[Convergence rate]
	\label{lem:convrate}
	If $U\in L^\infty(\R^+;W^{1,q}(\dom))\cap L^1(\R^+;W^{1,\infty}(\dom))$ for some $q\in (1,\infty]$ and $u_0\in H^1_{\Div}(\dom)$, then the approximations $u_N$ computed by the spectral scheme~\eqref{eq:numscheme} satisfy
	\begin{equation}
		\label{eq:convrate}
		\norm{u-u_N}_{L^2}(t)\leq C\norm{U-P_NU}_{L^1_t L^\infty_x}+ CN^{-1}\norm{u_0 }_{ H^1_x}\exp\left(\norm{U}_{L^1_t W^{1,\infty}_x}\right),
	\end{equation}
	where $\norm{U-P_NU}_{L^1_t L^\infty_x}$ goes to zero due to  the approximation properties of the Fourier truncation $P_N$.
	
\end{lemma}
\begin{proof}
We consider $u_N-\p_N u$, where $u\in L^\infty(0,T;H^1(\dom))$ is a weak solution of~\eqref{eq:vectortransport} and $u_N$ the approximation with the spectral scheme~\eqref{eq:numscheme} for the corresponding $U$ and initial data $u_0$. Due to the truncation, $\p_N u$ is smooth in space, and therefore satisfies the PDE pointwise:
\begin{equation*}
	\partial_t \p_N u = \p_N ((U\cdot \Grad)u).
\end{equation*}
Thus, the difference $u_N-\p_N u$ satisfies
\begin{equation*}
	\partial_t (u_N - \p_N u) = \p_N\left((U_N\cdot\Grad)u_N - (U\cdot \Grad ) u\right).
\end{equation*}
We multiply this equation by $u_N-\p_N u$, integrate, use chain rule and that $\p_N(u_N-\p_N u)=u_N-\p_N u$:
	\begin{align*}
	&\frac12\frac{d}{dt}\norm{\p_Nu-u_N}_{L^2}^2   = \int_{\dom} 	(u_N-\p_N u)\cdot \partial_t(u_N-\p_N u)dx  \\
	&\quad = \int_{\dom}\p_N\left((U_N\cdot \Grad)u_N -(U\cdot \Grad ) u\right)\cdot (u_N-\p_N u) dx\\
		&\quad = \int_{\dom}\left((U_N\cdot \Grad)u_N -(U\cdot \Grad ) u\right)\cdot (u_N-\p_N u) dx\\
	& \quad = \int_{\dom}(((U_N-U)\cdot \Grad)u)\cdot (u_N-\p_N u) dx  - \int_{\dom}((U_N\cdot \Grad)(u-u_N) )\cdot (u_N-\p_N u) dx \\
	& \quad \leq \norm{U-U_N}_{L^\infty}\norm{\Grad u}_{L^2} \norm{u_N-\p_N u}_{L^2} - \int_{\dom} ((U_N\cdot\Grad)(u-\p_N u))\cdot (u_N-\p_N u) dx \\
	&\qquad - \int_{\dom} ((U_N\cdot\Grad )(\p_N u - u_N))\cdot (u_N-\p_N u)dx \\
	& \quad \leq\norm{U-U_N}_{L^\infty}\norm{\Grad u}_{L^2} \norm{u_N-\p_N u}_{L^2} +\norm{U_N}_{L^\infty}\norm{u-\p_N u}_{L^2}\norm{u_N-\p_N u}_{L^2}   .
\end{align*}
Hence
\begin{equation*}
	\frac{d}{dt}\norm{\p_N u - u_N}_{L^2}\leq \norm{\Grad u}_{L^2}\norm{U-U_N}_{L^\infty}+ \norm{U_N}_{L^\infty}\norm{u-\p_Nu}_{L^2}.
\end{equation*}
Since $\p_N u(0)-u_N(0)=0$, and both $U$ and $u$ are divergence free, we obtain
\begin{align*}
	\norm{\p_Nu-u_N}_{L^2}&\leq C\left(\norm{U-\p_NU}_{L^1(0,t;L^\infty(\dom))}+ \norm{u-\p_Nu}_{L^\infty_tL^2_x}\right)\\
	&= C\left(\norm{U-P_NU}_{L^1(0,t;L^\infty(\dom))}+ \norm{u-P_Nu}_{L^\infty_tL^2_x}\right)
\end{align*}
and in particular,
\begin{align*}
	\norm{u-u_N}_{L^2}&\leq C\left(\norm{U-P_NU}_{L^1(0,t;L^\infty(\dom))}+\norm{u-P_Nu}_{L^\infty_tL^2_x}\right)\\
    & \leq C\left(\norm{U-P_NU}_{L^1(0,t;L^\infty(\dom))}+N^{-1}\norm{u }_{L^\infty_tH^1_x}\right)\\
    & \leq C\left(\norm{U-P_NU}_{L^1(0,t;L^\infty(\dom))}+N^{-1}\norm{u_0 }_{ H^1_x}\exp\left(\norm{U}_{L^1_t W^{1,\infty}_x}\right)\right),
\end{align*}
using the approximation properties of the Fourier projection $P_N$ and Lemma~\ref{lem:h1bound}.
The right hand side goes to zero at a rate depending on the smoothness of $U$ and $u$ as $n\to \infty$.

\end{proof}
\section{Mixing estimates}
\label{sec: Mixing estimates}

The goal of this section is to derive estimates on the minimal mixing time of weak solutions of~\eqref{eq:vectortransport} in the $H^{-\alpha}(\dom)$ norm for $\alpha\in [\frac12,1]$. All the estimates can be derived at the level of the approximations $u_N$ by the spectral scheme~\eqref{eq:numscheme}. As we shall see, these estimates are uniform in $N$ and thus we can take the limit $N\to\infty$.
For any $\alpha>0$, following~\cite{Lin2011}, let us define 
\begin{equation*}
	(-\Delta )^\alpha \varphi_N := \left(\frac{2\pi}{L}\right)^{2\alpha}\sum_{|k|_\infty\leq N} |k|^{2\alpha} \hat{\varphi}_k(t)e^{i2\pi \frac{k\cdot x}{L}} := u_N,
\end{equation*}
and
\begin{equation*}
	|\Grad|^\alpha \varphi_N : =  \left(\frac{2\pi}{L}\right)^{\alpha}\sum_{|k|_\infty\leq N}  |k|^{\alpha} \hat{\varphi}_k(t)e^{i2\pi \frac{k\cdot x}{L}}
\end{equation*}
Therefore, $\hat{\varphi}_k(t) = \frac{1}{|k|^{2\alpha}} \left(\frac{L}{2\pi}\right)^{\alpha} \hat{u}_{k}(t)$ and 
\begin{equation*}
	\norm{u_N}_{H^{-\alpha}(\dom)} = \norm{|\Grad|^\alpha \varphi_N}_{L^2(\dom)}.
\end{equation*}
Next, we make the following computation for the time evolution of $\norm{u_N}_{H^{-\alpha}}$ using the PDE \eqref{eq:numscheme}:
\begin{align*}
	\frac{1}{2}\frac{d}{dt}\norm{u_N}_{H^{-\alpha}}^2 & = \frac12\frac{d}{dt}\norm{|\Grad|^\alpha \varphi_N}_{L^2}^2, \\
	& = \int_{\dom} \varphi_N\cdot\partial_t u_N dx, \\
	& = -\int_{\dom} \p_N((U_N\cdot\Grad)u_N)\cdot \varphi_N dx, \\
	& = -\int_{\dom} ((U_N\cdot\Grad)u_N)\cdot \varphi_N dx, \\
	& = -\int_{\dom} ((U_N\cdot\Grad)(-\Delta)^\alpha \varphi_N)\cdot \varphi_N dx,\\
	& = \int_{\dom} ((|\Grad|^\alpha U_N\cdot\Grad)|\Grad |^\alpha \varphi_N)\cdot \varphi_N dx,\\
	& = -\int_{\dom} ((|\Grad|^\alpha U_N\cdot\Grad)\varphi_N): |\Grad |^\alpha\varphi_N dx.
\end{align*}
We therefore obtain the following lower bound:
\begin{equation}\label{eq:tia}
	\begin{split}
		\frac12\frac{d}{dt}\norm{u_N}_{H^{-\alpha}}^2 \geq - \norm{|\Grad|^\alpha U_N}_{L^r}\norm{|\Grad|^\alpha \varphi_N}_{L^{\frac{2r}{r-2}}}\norm{\Grad \varphi_N}_{L^2},
	\end{split}
\end{equation}
where $r\geq 1$ is to be determined such that $\norm{|\Grad|^\alpha U_N}_{L^r}$ is uniformly bounded with respect to $N$. We will distinguish the three cases $q(1-\alpha)<d$, $q(1-\alpha)=d$ and $q(1-\alpha)>d$, since they will require either a fractional version of the Gagliardo-Nirenberg-Sobolev inequality or Morrey's inequality. We will use $C$ to denote a constant that can vary from line to line and may depend on $\alpha,d$ and $q$ but not on $N$ or time $t$. 
\subsection{Case $q(1-\alpha)<d$ }
This includes the case $\alpha=1$ if $q\neq \infty$.
\begin{lemma}\label{lem:case1}
	Assume that either $q(1-\alpha)<d$, or else $\alpha=1$ and $q\neq\infty$. Then, we have for any $t\geq 0$,
	\begin{equation*}
		\norm{u_N(t)}_{H^{-\alpha}}\geq \left[\norm{u_N(0)}_{H^{-\alpha}}^{\frac{d}{\alpha q}} - t \frac{C}{\alpha} \norm{\Grad U_N}_{L^\infty_tL^q_x}\norm{u_N(0)}_{L^2}^{\frac{d}{\alpha q}}\right]^{\frac{\alpha q}{d}}.
	\end{equation*}
	In particular, we have for any limit $u$ of the sequence $u_N$, 
		\begin{equation*}
		\norm{u(t)}_{H^{-\alpha}}\geq \left[\norm{u_0}_{H^{-\alpha}}^{\frac{d}{\alpha q}} - t \frac{C}{\alpha} \norm{\Grad U}_{L^\infty_tL^q_x}\norm{u_0}_{L^2}^{\frac{d}{\alpha q}}\right]^{\frac{\alpha q}{d}}.
	\end{equation*}
	Thus the minimal mixing time is given by
	\begin{equation*}
		t_{\min} =   \frac{ \alpha\norm{u_0}_{H^{-\alpha}}^{\frac{d}{\alpha q}}}{ C\norm{u_0}_{L^2}^{\frac{d}{\alpha q}} \norm{\Grad U}_{L^\infty_t L^{q}_x}  }.
	\end{equation*}
	\end{lemma}
\begin{proof}
We start with~\eqref{eq:tia}.
To find $r$, we use the fractional Sobolev inequality~\cite[Thm. 6.5]{DiNezza2012}, or Theorem~\ref{thm:fractionalsobolev} in the Appendix, with $f= (|\Grad|^\alpha U_N)_i$, $i=1,\dots, d$,
\begin{equation}\label{eq:Ur}
	\norm{|\Grad|^\alpha U_N}_{L^r}\leq C  \||\Grad|^\alpha U_N\|_{W^{1-\alpha,q}} = C  \norm{\Grad U_N}_{L^q},
\end{equation}
for $r = \frac{dq}{d-(1-\alpha)q}$, where we need $q(1-\alpha)<d$, and this estimate is independent of $N\in\N$ due to the assumption~\eqref{eq:Ubound}.
We note that if $\alpha<\frac12$, the gradient of $\varphi$ is  a positive order Sobolev norm of $u$ for which we may not have uniform bounds unless $U$ also has higher regularity. As we assume that $\alpha\geq\frac12$, we can use  the fractional Sobolev interpolation inequality~\cite{Brezis2018}, stated in the Appendix, Theorem~\ref{thm:fractionalinterpolation} for $\norm{\Grad\varphi_N}_{L^2}$, using $s_1=\alpha$, $s_2=2\alpha$, $p=p_1=p_2=2$ and $\theta = \frac{2\alpha-1}{\alpha}$, we get
\begin{equation}\label{eq:fractionalL2}
	\norm{\Grad\varphi_N}_{L^2}\leq C\norm{\varphi_N}_{W^{\alpha,2}}^{\frac{2\alpha-1}{\alpha}} \norm{\varphi_N}_{W^{2\alpha,2}}^{\frac{1-\alpha}{\alpha}} = C \norm{u_N}_{H^{-\alpha}}^{\frac{2\alpha-1}{\alpha}}\norm{u_N(0)}_{L^2}^{\frac{1-\alpha}{\alpha}}.
\end{equation}
For the term $\norm{|\Grad|^\alpha \varphi_N}_{L^{\frac{2r}{r-2}}}$ in~\eqref{eq:tia}, we  use the fractional Gagliardo-Nirenberg-Sobolev inequality~\cite[Thm 1]{Brezis2019}, restated in the appendix in Theorem~\ref{thm:fractionalGNS}. Taking in that theorem $\tilde{r}=0$, $s_1=0$, $s_2=\alpha$, $p_1=p_2=2$, $\tilde{q}=\frac{2r}{r-2}$, none of the exception cases state there apply, and thus we have 

\begin{equation}\label{eq:otherterm}
	\norm{|\Grad|^\alpha \varphi_N}_{L^{\tilde{q}}}\leq C \norm{|\Grad|^\alpha\varphi_N}_{L^2}^\theta \norm{(-\Delta)^\alpha\varphi_N}_{L^2}^{1-\theta},
\end{equation} 
where $\theta$ is given by 
\begin{equation*}
	\theta = 1 +\frac{d}{\alpha}\left(\frac1{\tilde{q}}-\frac12\right) = \frac{2\tilde{q}\alpha + 2d-\tilde{q}d}{2\tilde{q}\alpha}.
\end{equation*}

The value of $\theta$ was obtained using the following:
\begin{equation*}
	\frac{1}{\tilde{q}} = \frac12-\frac{s}{d},\quad s = (1-\theta)\alpha.
\end{equation*}
Substituting  $\tilde{q}=\frac{2r}{r-2}$, we get the beautiful expression:
\begin{equation}\label{eq:howtogettheta}
	\theta = 1 +\frac{d}{\alpha}\left(\frac{r-2}{2r}-\frac12\right) =1-\frac{d}{\alpha r}= \frac{\alpha r-d}{\alpha r}.
\end{equation}
Next, we use the expression we derived for $r$ above:
\begin{equation*}
	\theta = \frac{1}{\alpha }-\frac{d}{\alpha q}= \frac{q-d}{\alpha q}.
\end{equation*}
We plug~\eqref{eq:Ur}, \eqref{eq:fractionalL2}, and~\eqref{eq:otherterm} into~\eqref{eq:tia} to obtain
\begin{equation*}
	\begin{split}
		\frac12\frac{d}{dt}\norm{u_N}_{H^{-\alpha}}^2& \geq -C \norm{\Grad U_N}_{L^q} \norm{u_N}_{H^{-\alpha}}^{\frac{2\alpha-1}{\alpha}+\frac{q-d}{\alpha q}}\norm{u_N(0)}_{L^2}^{\frac{1-\alpha}{\alpha}+ \frac{\alpha q-q+d}{\alpha q}}\\
		& \geq -C \norm{\Grad U_N}_{L^q} \norm{u_N}_{H^{-\alpha}}^{2-\frac{d}{\alpha q}}\norm{u_N(0)}_{L^2}^{\frac{d}{\alpha q}}.
	\end{split}
\end{equation*}
As we have $\frac{d}{\alpha q} > 0$, so we can do the following manipulation:
\begin{equation*}
	\frac{d}{dt}\norm{u_N}_{H^{-\alpha}}^{\frac{d}{\alpha q}}\geq - \frac{d}{\alpha q} C  \norm{\Grad U_N}_{L^q} \norm{u_N(0)}_{L^2}^{\frac{d}{\alpha q}},
\end{equation*}
which integrates to
\begin{equation*}
	\norm{u_N(t)}_{H^{-\alpha}}\geq \left[\norm{u_N(0)}_{H^{-\alpha}}^{\frac{d}{\alpha q}}- \frac{d}{\alpha q} C  \int_0^t\norm{\Grad U_N(s)}_{L^q} ds \norm{u_N(0)}_{L^2}^{\frac{d}{\alpha q}}\right]^{\frac{\alpha q}{d}}.
\end{equation*}
Since we have $\Grad U \in L^\infty(0,\infty;L^q(\dom))$, above expression simplifies to
\begin{equation}\label{eq:minimalmixingcase1}
	\norm{u_N(t)}_{H^{-\alpha}}\geq \left[\norm{u_N(0)}_{H^{-\alpha}}^{\frac{d}{\alpha q}}- t\frac{C}{\alpha }   \norm{\Grad U_N}_{L^\infty_tL^q_x} \norm{u_N(0)}_{L^2}^{\frac{d}{\alpha q}}\right]^{\frac{\alpha q}{d}}.
\end{equation}
Now $u_N(0)\to u_0$ in $L^2(\dom)$ and $\Grad U_N\to \Grad U$ in $L^\infty([0,T],L^q(\dom))$ for any $T>0$, as $N\to \infty$ and using that $L^2$ is compactly embedded in $H^{-\alpha}$ and the time continuity estimate from Lemma~\ref{lem:existence}, we can use the Aubin-Lions lemma~\cite{Simon1987} to obtain that $u_N\to u$ in $C([0,T];H^{-\alpha}(\dom))$ up to a subsequence as $N\to\infty$ and thus we have in the limit:
\begin{equation}\label{eq:minimalmixingcase1limit}
	\norm{u(t)}_{H^{-\alpha}}\geq \left[\norm{u_0}_{H^{-\alpha}}^{\frac{d}{\alpha q}}- t\frac{C}{\alpha}   \norm{\Grad U}_{L^\infty_tL^q_x} \norm{u_0}_{L^2}^{\frac{d}{\alpha q}}\right]^{\frac{\alpha q}{d}}.
\end{equation}

Thus the minimal mixing time is given by
\begin{equation*}
	t_{\min} =  \frac{\alpha\norm{u_0}_{H^{-\alpha}}^{\frac{d}{\alpha q}}}{ C\norm{u_0}_{L^2}^{\frac{d}{\alpha q}} \norm{\Grad U}_{L^\infty_t L^{q}_x}  }.
\end{equation*}
\end{proof}
\subsection{Case $q(1-\alpha)>d$ }
This includes the case that $q=\infty$ when $\alpha<1$.
\begin{lemma}\label{lem:case2}
	Let $q$, $\alpha$, and $d$ satisfy either $q(1-\alpha)>d$, or $q=\infty$ and $\alpha<1$. Then, we have for any $t\geq 0$,
\begin{equation*}
	\norm{u_N(t)}_{H^{-\alpha}}\geq \left[ \norm{u_N(0)}_{H^{-\alpha}}^{\frac{1-\alpha}{\alpha}} - t \frac{1-\alpha}{\alpha} C\norm{u_N(0)}_{L^2}^{\frac{1-\alpha}{\alpha}}\norm{U_N}_{L^\infty_t W^{1,q}_x} \right]^{\frac{\alpha}{1-\alpha}}.
\end{equation*}
	In particular, we have for any limit $u$ of the sequence $u_N$, 
\begin{equation*}
	\norm{u(t)}_{H^{-\alpha}}\geq \left[ \norm{u_0}_{H^{-\alpha}}^{\frac{1-\alpha}{\alpha}} - t \frac{1-\alpha}{\alpha} C\norm{u_0}_{L^2}^{\frac{1-\alpha}{\alpha}}\norm{U}_{L^\infty_t W^{1,q}_x} \right]^{\frac{\alpha}{1-\alpha}}.
\end{equation*}
	Thus the minimal mixing time is given by
	\begin{equation*}
		t_{\min}= \frac{\alpha\norm{u_0}_{H^{-\alpha}}^{\frac{1-\alpha}{\alpha}}}{(1-\alpha) C\norm{u_0}_{L^2}^{\frac{1-\alpha}{\alpha}} \norm{U}_{L^\infty_t W^{1,q}_x}  }.
	\end{equation*}
\end{lemma}

\begin{proof}

In this case, we obtain that $|\Grad|^\alpha U$ is actually H\"older continuous by~\cite[Thm. 8.2]{DiNezza2012}, restated in the Appendix as Theorem~\ref{thm:fractionalmorrey}. We first note that
\begin{equation*}
	\norm{|\Grad|^\alpha U}_{L^\infty} \leq C \norm{|\Grad|^\alpha U}_{W^{1-\alpha,q}}\approx C \norm{U}_{W^{1,q}},
\end{equation*}
and the same inequality holds true for $U_N$ uniformly in $N$. 
From ~\eqref{eq:tia}, we, therefore, obtain:
\begin{equation*}
	\begin{split}
	\frac12\frac{d}{dt}\norm{u_N}^2_{H^{-\alpha}}&\geq -\norm{|\Grad|^\alpha U_N}_{L^\infty}\norm{|\Grad|^\alpha \varphi_N}_{L^2}\norm{\Grad \varphi_N}_{L^2}\\
	&\geq - C \norm{U_N}_{W^{1,q}}\norm{u_N}_{H^{-\alpha}}\norm{\Grad\varphi_N}_{L^2}.
	\end{split}
\end{equation*}
Next, applying the fractional interpolation inequality~\eqref{eq:fractionalGN} with $p = p_1=p_2=2$ and $\theta=\frac{2\alpha-1}{\alpha}$, we find that
\begin{equation*}
	\norm{\Grad\varphi_N}_{L^2}\leq C \norm{\varphi_N}_{W^{\alpha,2}}^{\frac{2\alpha-1}{\alpha}} \norm{\varphi_N}_{W^{2\alpha,2}}^{\frac{1-\alpha}{\alpha}} = C \norm{u_N}_{H^{-\alpha}}^{\frac{2\alpha-1}{\alpha}}\norm{u_N(0)}_{L^2}^{\frac{1-\alpha}{\alpha}}.
\end{equation*}
Substituting this into the previous equation, we get
\begin{equation*}
	\frac12\frac{d}{dt}\norm{u_N}^2_{H^{-\alpha}}\geq - C \norm{U_N}_{W^{1,q}}\norm{u_N}_{H^{-\alpha}}^{\frac{3\alpha-1}{\alpha}}\norm{u_N(0)}_{L^2}^{\frac{1-\alpha}{\alpha}}.
\end{equation*}
Note that $1-\frac{3\alpha-1}{\alpha}=\frac{1-2\alpha}{\alpha}>-1$ since $\alpha<1$, thus the following manipulation is valid:
\begin{equation*}
	\frac{d}{dt}\norm{u_N}^{\frac{1-\alpha}{\alpha}}_{H^{-\alpha}}\geq -\frac{1-\alpha}{\alpha} C \norm{U_N}_{W^{1,q}}\norm{u_N(0)}_{L^2}^{\frac{1-\alpha}{\alpha}}.
\end{equation*}
After integrating and taking roots, we get
\begin{equation*}
	\norm{u_N(t)}_{H^{-\alpha}}\geq \left[ \norm{u_N(0)}_{H^{-\alpha}}^{\frac{1-\alpha}{\alpha}} -  \frac{1-\alpha}{\alpha} C \norm{u_N(0)}_{L^2}^{\frac{1-\alpha}{\alpha}}\int_0^t\norm{U_N(s)}_{W^{1,q}}ds \right]^{\frac{\alpha}{1-\alpha}}.
\end{equation*}
Since we assume that the gradient of $U$ is bounded in the $L^q$ norm uniformly in time, we obtain
\begin{equation}\label{eq:minimalmixing2N}
	\norm{u_N(t)}_{H^{-\alpha}}\geq \left[ \norm{u_N(0)}_{H^{-\alpha}}^{\frac{1-\alpha}{\alpha}} - t \frac{1-\alpha}{\alpha} C \norm{u_N(0)}_{L^2}^{\frac{1-\alpha}{\alpha}}\norm{U_N}_{L^\infty_t W^{1,q}_x} \right]^{\frac{\alpha}{1-\alpha}}.
\end{equation}
As in the previous lemma, we use that $u_N(0)\to u_0$ in $L^2$ (and thus by the continuity of the embedding of $L^2$ in $H^{-\alpha}$) and that $U_N\to U$ and that $u_N\to u$ in $C(0,T;H^{-\alpha}(\dom))$ up to a subsequence for any $T>0$ by the Aubin--Lions lemma, to obtain:
\begin{equation}\label{eq:minimalmixing2}
	\norm{u(t)}_{H^{-\alpha}}\geq \left[ \norm{u_0}_{H^{-\alpha}}^{\frac{1-\alpha}{\alpha}} - t \frac{1-\alpha}{\alpha} C \norm{u_0}_{L^2}^{\frac{1-\alpha}{\alpha}}\norm{U}_{L^\infty_t W^{1,q}_x} \right]^{\frac{\alpha}{1-\alpha}}.
\end{equation}
This yields a minimal mixing time
\begin{equation*}
	t_{\min}= \frac{\alpha\norm{u_0}_{H^{-\alpha}}^{\frac{1-\alpha}{\alpha}}}{(1-\alpha) C \norm{u_0}_{L^2}^{\frac{1-\alpha}{\alpha}} \norm{U}_{L^\infty_t W^{1,q}_x}  }.
\end{equation*}
\end{proof}

\subsection{Case $q(1-\alpha)=d$ }
\begin{lemma}\label{lem:case3}
	Let $q$, $\alpha$ and $d$ satisfy $q(1-\alpha)=d$. Then, we have for any $t\geq 0$,
	\begin{equation*}
			\norm{u_N(t)}_{H^{-\alpha}}\geq \left[\norm{u_N(0)}_{H^{-\alpha}}^{\frac{1-\alpha}{\alpha} +\frac{d}{\alpha r}}- \left(\frac{1-\alpha}{\alpha}+\frac{d}{\alpha r}\right)  C_r  t \norm{\Grad U_N}_{L^\infty_tL^q_x}  \norm{u_N(0)}_{L^2}^{\frac{1-\alpha}{\alpha } + \frac{d}{\alpha r}}\right]^{\frac{\alpha r}{r-\alpha r +d}},
	\end{equation*}
	where $r$ can be chosen arbitrarily large.
	In particular, we have for any limit $u$ of the sequence $u_N$, 
	\begin{equation*}
		\norm{u(t)}_{H^{-\alpha}}\geq \left[\norm{u_0}_{H^{-\alpha}}^{\frac{1-\alpha}{\alpha} +\frac{d}{\alpha r}}- \left(\frac{1-\alpha}{\alpha}+\frac{d}{\alpha r}\right)  C_r  t \norm{\Grad U}_{L^\infty_tL^q_x}  \norm{u_0}_{L^2}^{\frac{1-\alpha}{\alpha } + \frac{d}{\alpha r}}\right]^{\frac{\alpha r}{r-\alpha r +d}},
	\end{equation*}
	for any $r\in [1,\infty)$.
	
	Thus the minimal mixing time is given by
	\begin{equation*}
	t_{\min} = \frac{\alpha r\norm{u_0}_{H^{-\alpha}}^{\frac{1-\alpha}{\alpha} +\frac{d}{\alpha r}}}{(r-\alpha r +d) C_r\norm{u_0}_{L^2}^{\frac{1-\alpha}{\alpha } + \frac{d}{\alpha r}} \norm{\Grad U}_{L^\infty_t L^{q}_x}  },
	\end{equation*}
\end{lemma}

\begin{proof}

In this cases, we obtain
\begin{equation}\label{eq:sobolevedgecase}
	\norm{|\Grad|^\alpha U}_{L^r}\leq C_r \norm{\Grad U}_{L^q},
\end{equation}
for any $r\in [1,\infty)$, see~\cite[Thm 6.9]{DiNezza2012}. Similar to the computations done above in Lemma~\ref{lem:case1}, c.f. Equation~\eqref{eq:howtogettheta}, we obtain 
\begin{equation*}
	\frac12\frac{d}{dt}\norm{u}_{H^{-\alpha}}^2\geq -C_r \norm{\Grad U}_{L^q}\norm{u}_{H^{-\alpha}}^{\frac{2\alpha-1}{\alpha}+1-\frac{d}{\alpha r}}\norm{u_0}_{L^2} ^{\frac{1-\alpha}{\alpha } + \frac{d}{\alpha r}}
\end{equation*}
We manipulate this to obtain
\begin{equation*}
	\frac{d}{dt}\norm{u}_{H^{-\alpha}}^{\frac{1-\alpha}{\alpha} +\frac{d}{\alpha r}}\geq - \left(\frac{1-\alpha}{\alpha}+\frac{d}{\alpha r}\right) C_r \norm{\Grad U}_{L^q} \norm{u_0}_{L^2}^{\frac{1-\alpha}{\alpha } + \frac{d}{\alpha r}},
\end{equation*}
which  integrates to 
\begin{equation*}
	\norm{u(t)}_{H^{-\alpha}}\geq \left[\norm{u_0}_{H^{-\alpha}}^{\frac{1-\alpha}{\alpha} +\frac{d}{\alpha r}}- \left(\frac{1-\alpha}{\alpha}+\frac{d}{\alpha r}\right)  C_r \int_0^t\norm{\Grad U(s)}_{L^q} ds \norm{u_0}_{L^2}^{\frac{1-\alpha}{\alpha } + \frac{d}{\alpha r}}\right]^{\frac{\alpha r}{r-\alpha r +d}}.
\end{equation*}
Now using that $\Grad U \in L^\infty(0,\infty;L^q(\dom))$, above bound simplifies to
\begin{equation}\label{eq:minimalmixingcasex}
	\norm{u(t)}_{H^{-\alpha}}\geq \left[\norm{u_0}_{H^{-\alpha}}^{\frac{1-\alpha}{\alpha} +\frac{d}{\alpha r}}- \left(\frac{1-\alpha}{\alpha}+\frac{d}{\alpha r}\right)  C_r  t \norm{\Grad U}_{L^\infty_tL^q_x}  \norm{u_0}_{L^2}^{\frac{1-\alpha}{\alpha } + \frac{d}{\alpha r}}\right]^{\frac{\alpha r}{r-\alpha r +d}}.
\end{equation}
where $r$ can be chosen arbitrarily large but the constant $C_r$ depends on $r$ due to the Sobolev embedding~\eqref{eq:sobolevedgecase}. 
The minimal mixing time is given by 
\begin{equation*}
	t_{\min} = \frac{\alpha r\norm{u_0}_{H^{-\alpha}}^{\frac{1-\alpha}{\alpha} +\frac{d}{\alpha r}}}{(r-\alpha r +d) C_r\norm{u_0}_{L^2}^{\frac{1-\alpha}{\alpha } + \frac{d}{\alpha r}} \norm{\Grad U}_{L^\infty_t L^{q}_x}  },
\end{equation*}
where again $r$ can be chosen arbitrarily large.
\end{proof}
\subsection{Case $q=\infty$, $\alpha=1$}
Finally, we treat the special case where $q=\infty$ and $\alpha=1$. 
In this case, we have
\begin{align}
	\frac12\frac{d}{dt}\norm{u_N}^2_{H^{-1}(\dom)}\geq -\norm{\Grad\varphi_N}_{L^{2}(\dom)}^2\norm{\Grad U_N}_{L^\infty(\dom)}=-\norm{u_N}^2_{H^{-1}} \norm{\Grad U_N}_{L^\infty(\dom)}
\label{eq:Linfty}
\end{align}
thus with Gr\"onwall's inequality, we get
\begin{equation*}
	\norm{u_N(t)}_{H^{-1}}^2 \geq \norm{u_N(0)}_{H^{-1}}^2 \exp\left(-2\int_0^t\norm{\Grad U_N(s)}_{L^\infty} ds \right).
\end{equation*}
Sending $N\to \infty$ (note that the limit is unique in this case by Lemma~\ref{lem:stability}), we obtain 
\begin{equation*}
	\norm{u(t)}_{H^{-1}}^2 \geq \norm{u_0}_{H^{-1}}^2 \exp\left(-2\int_0^t\norm{\Grad U(s)}_{L^\infty} ds \right).
\end{equation*}

\section{Numerical Algorithm and Results}
\label{sec: Numerical algorithm}
In this section, we perform numerical computations for mixing in two dimensions and provide evidence that an exponential mixing rate is achievable in the specific case \(q=2\), that is, under the constraint \(\|U\|_{W^{1,2}} \leq 1\). We work on a torus $\mathbb{T}^2$ of side length \(L=2\pi\). In this setting, we provide a numerical indication that the optimal mixing rate is at least exponentially fast.

Our approach is to choose \(U\) so as to maximize, instantaneously, the decay of the \({H}^{-\alpha}\)-norm of \(u\). This idea was previously considered by Lin, Thiffeault and Doering \cite{Lin2011} in the context of scalar mixing. Multiplying the governing equation by \(|\nabla|^{-2\alpha}u\) and integrating by parts yields
\begin{align}
\frac{1}{2}\frac{d}{dt}\|u\|_{{H}^{-\alpha}}^2
=
\int_{\mathbb{T}^d}
U \cdot \mathbb{P}\big(\nabla |\nabla|^{-2\alpha}u \cdot u\big)\,{\rm d}x.
\end{align}
Therefore, in order to maximize the instantaneous decay of \(\|u\|_{{H}^{-\alpha}}\) while satisfying the constraint \(\|U\|_{W^{1,2}} \leq 1\), we choose
\begin{align}
U
=
-
\frac{(-\Delta)^{-1}\mathbb{P}\big(\nabla |\nabla|^{-2\alpha}u \cdot u\big)}
{\big\langle \big|\nabla (-\Delta)^{-1}\mathbb{P}\big(\nabla |\nabla|^{-2\alpha}u \cdot u\big)\big|^2 \big\rangle^{1/2}}.
\label{eqn: optimal U}
\end{align}
Here, \(\langle \cdot \rangle\) denotes integration over the torus and $\mathbb{P}$ is the Leray projector. Although the choice \eqref{eqn: optimal U} need not be the global optimizer, it provides a natural and computationally convenient strategy to obtain an upper bound on the optimal mixing rate.

Next, we apply the Fourier truncation for the velocity fields described in \Cref{sec: existence}, which leads to the system \eqref{eq:numscheme}. This time, in \eqref{eq:numscheme}, \(U_N\) denotes the projection of \eqref{eqn: optimal U} onto the first \(N\) Fourier modes. This yields a system of ODEs for the Fourier coefficients \(\hat u_k\), which we evolve in time using an adaptive RK45 scheme. In our computations, we take \(N=562\), corresponding to \(2N+1=1125\) modes in each spatial direction.

Using the Cartesian coordinates $(x,y)$ in this section, we consider an initial condition \(u_0\) defined through a stream function as
\[
u_0 \coloneqq \nabla^\perp \psi_0 \coloneqq (-\partial_y \psi_0,\partial_x \psi_0),
\]
where we choose $\psi_0$ so that \(u_0\) resembles a dipole flow. Specifically, we take
\begin{align}
\psi_0(x, y) = p_1 \left(\frac{x}{2 \pi}\right) p_2\left(\frac{y}{2 \pi}\right), \quad \text{where} \quad p_1(z) = 256 z^4 (1-z)^4 \text{ and } p_2(z) = \frac{8192}{27} z^3 (1-z)^3 (1 - 2 z). 
\end{align}

We now discuss our numerical observations, beginning with the case \(\alpha=1\), i.e., mixing is quantified through the decay of the \({H}^{-1}\)-norm of \(u\). A video of the simulation can be found here: 
\href{https://youtu.be/WI3ofKWohTw}{youtu.be/WI3ofKWohTw}.

\begin{figure}[htbp]
    \centering
    \includegraphics[width=\textwidth]{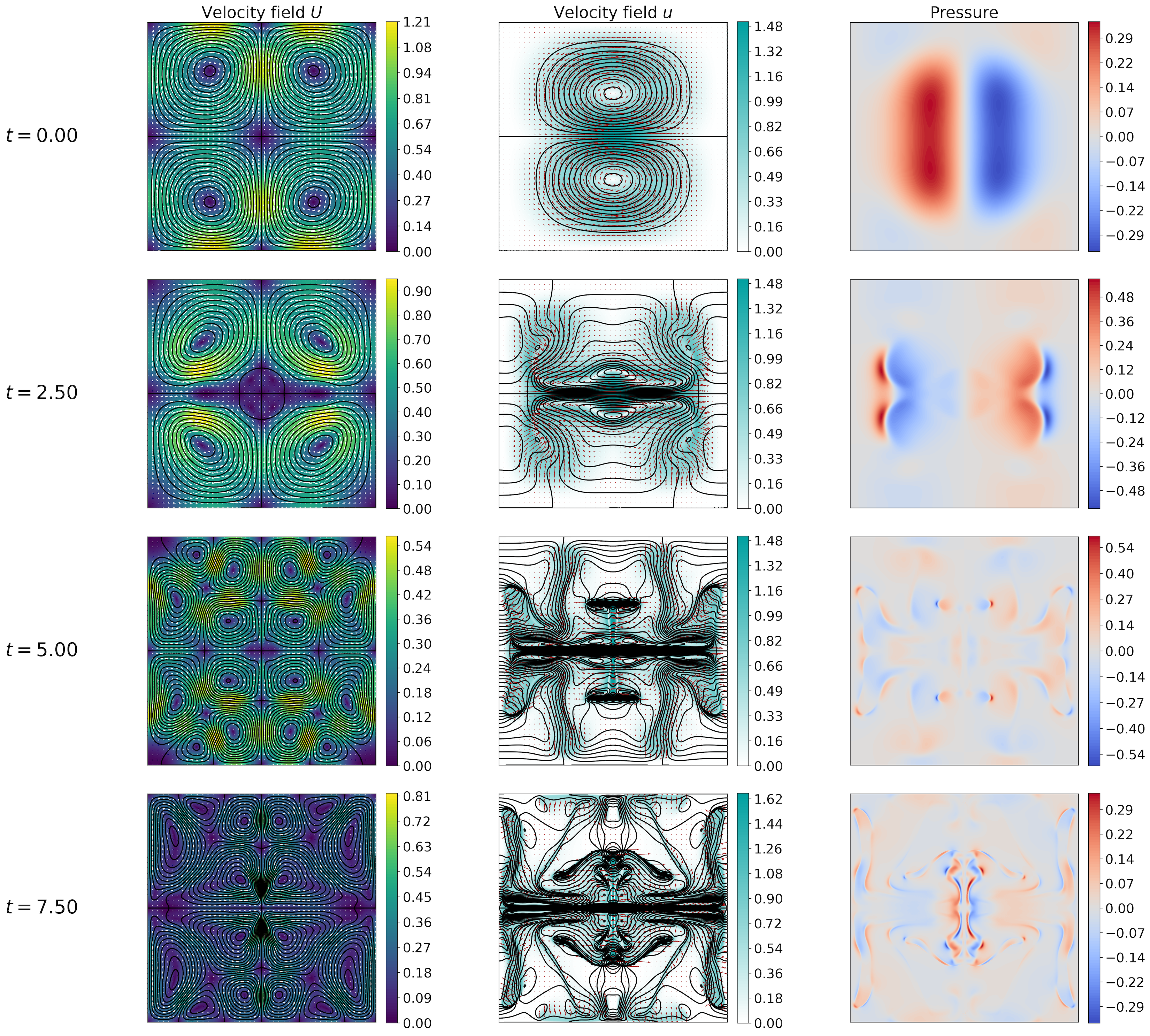}
    \caption{Snapshots of \(U\) (column 1), \(u\) (column 2) and \(p\) (column 3) at various times in the case \(\alpha = 1\). In columns 1 and 2, the colorbars represent the magnitudes of the velocity fields \(U\) and \(u\), respectively, while in column 3 the colorbar represents the pressure \(p\). The plots of \(U\) and \(u\) are overlaid with streamlines and arrows indicating the direction of the flow.}
    \label{fig:panel one}
\end{figure}

\begin{figure}[htbp]
    \centering
    \includegraphics[width=\textwidth]{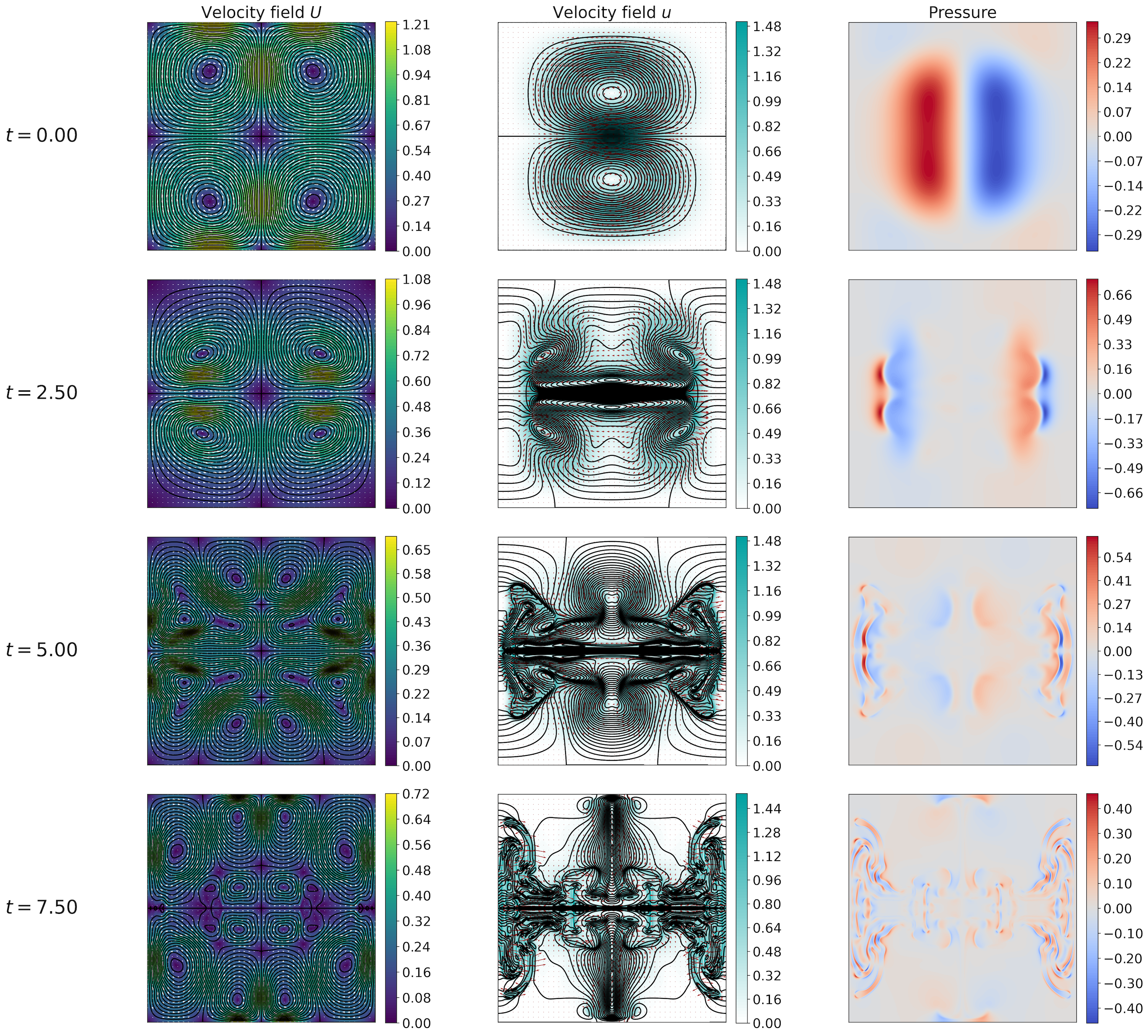}
    \caption{Same as \Cref{fig:panel one}, but for \(\alpha=\tfrac12\).}
    \label{fig:panel half}
\end{figure}

By construction, the initial velocity field \(u\) is a dipole. For this choice of \(u\), the vector field \(U\) determined through \eqref{eqn: optimal U} at time \(t=0\) takes the form of a counter-rotating cellular flow with four cells, while the pressure exhibits a bipolar structure as in \Cref{fig:panel one}. The role of this initial field \(U\) is to squeeze the streamlines of \(u\) toward the midplane and thereby generate a strong central shear layer. As time evolves, finer structures in \(u\) emerge through a repeated stretching and folding process. More precisely, the velocity field \(u\) becomes concentrated on increasingly thin filaments, resembling long and narrow shear layers. Away from these concentrated layers, \(u\) appears to be very small and on a large portion of the domain it is nearly stationary.

This feature of vector mixing is in stark contrast with scalar mixing. In the scalar case, long filaments may also form \cite{Lin2011}, but the scalar cannot concentrate on increasingly small subsets of the torus, since all \(L^p\)-norms, for \(p\in[1,\infty]\), are conserved. In the present vector setting, however, aside from the \(L^2\)-norm (energy), there is no analogous conservation law for the higher or lower \(L^p\)-norms. Hence concentration of the vector field \(u\) is, in principle, possible. This is indeed what we observe numerically.

The evolution of \(U\) is qualitatively different from that of \(u\). Although small scales also develop in \(U\), they do not appear as rapidly as in \(u\). Moreover, the structure of \(U\) remains predominantly cellular flows. In particular, \(U\) does not exhibit the same concentration phenomenon seen in \(u\). Instead, the thin shear layers in \(u\) form precisely along the interfaces between the counter-rotating cells of \(U\). Finally, the pressure \(p\) becomes even more concentrated as time evolves, especially in regions where the streamlines of \(u\) bend or turn around. Throughout the evolution, the norm \(\|u\|_{{H}^{-1}}(t)\) decays smoothly, in a manner consistent with an approximately exponential mixing regime (see \Cref{fig:norm one}).

Next, we briefly discuss the case \(\alpha=\tfrac12\), that is, when mixing is quantified through the decay of the \({H}^{-1/2}\)-norm of \(u\). \Cref{fig:panel half} shows a few snapshots and a complete video of the simulation can be found here:
\href{https://youtu.be/ykZvo9R21-Y}{youtu.be/ykZvo9R21-Y}. For an initial period of the evolution, up to approximately \(t=2\), the dynamics in the two cases \(\alpha=1\) and \(\alpha=\tfrac12\) appear qualitatively similar. Beyond this point, however, clear differences begin to emerge. 
For instance, around time \(t=3.3\), the \(\alpha=1\) evolution develops a filament along the vertical centerline, whereas no such filament development is observed in the \(\alpha=\tfrac12\) case.

In the \(\alpha=\tfrac12\) regime, small scales in \(u\) still develop, but they are distributed over a wider region of the domain and the shear layers are noticeably less concentrated. Correspondingly, the stretching and folding mechanism appears to be weaker than in the case \(\alpha=1\). Nevertheless, this evolution still appears to exhibit approximately exponential mixing, although with a smaller exponential rate (see \Cref{fig:norm half}). This is consistent with the fact that \( H^{-1}\) places comparatively less emphasis on smaller scales than \( H^{-1/2}\).

\begin{figure}[htbp]
    \centering
    \begin{subfigure}{0.49\textwidth}
        \centering
        \includegraphics[width=\textwidth]{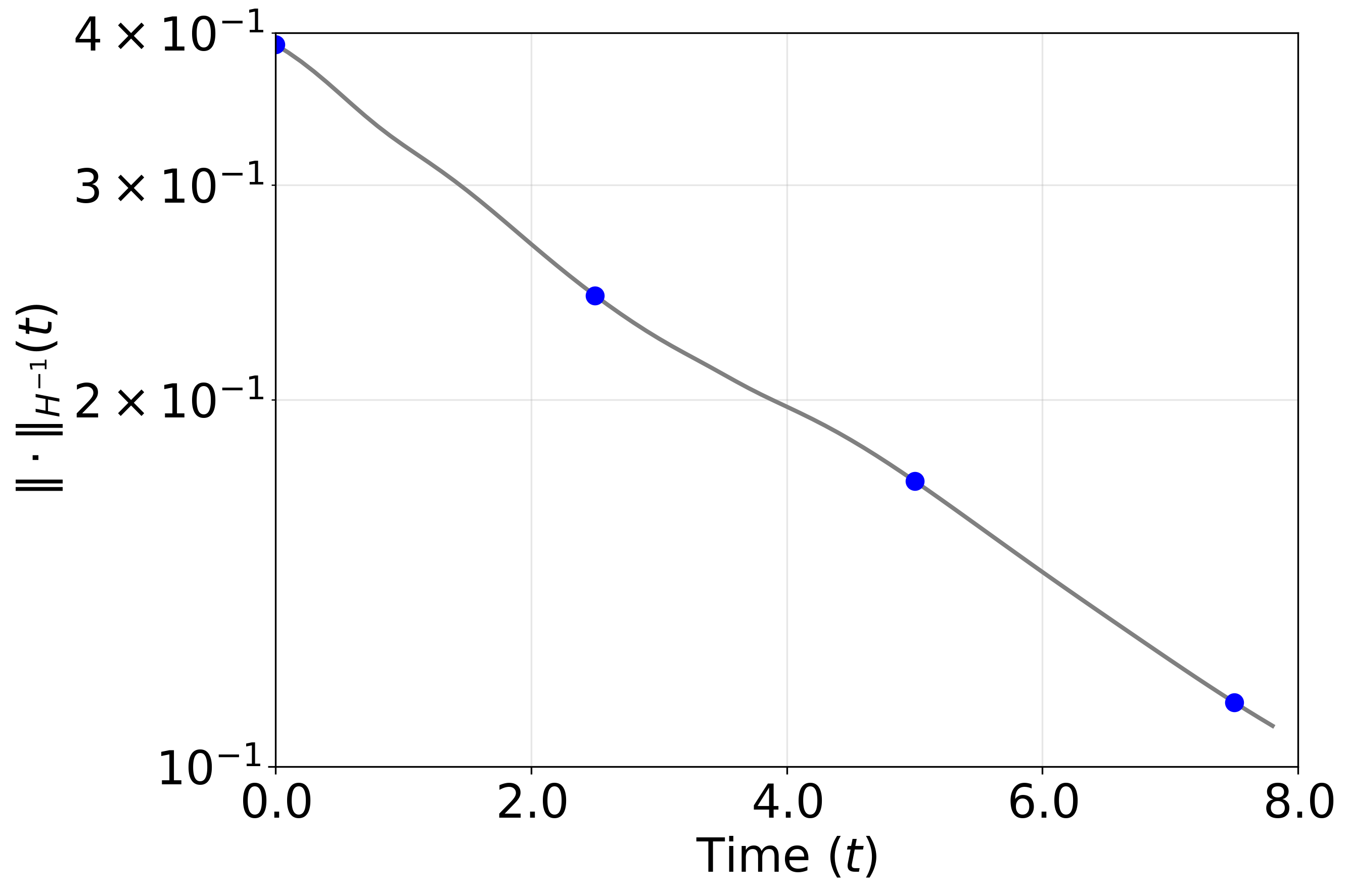}
        \caption{}
        \label{fig:norm one}
    \end{subfigure}
    \hfill
    \begin{subfigure}{0.49\textwidth}
        \centering
        \includegraphics[width=\textwidth]{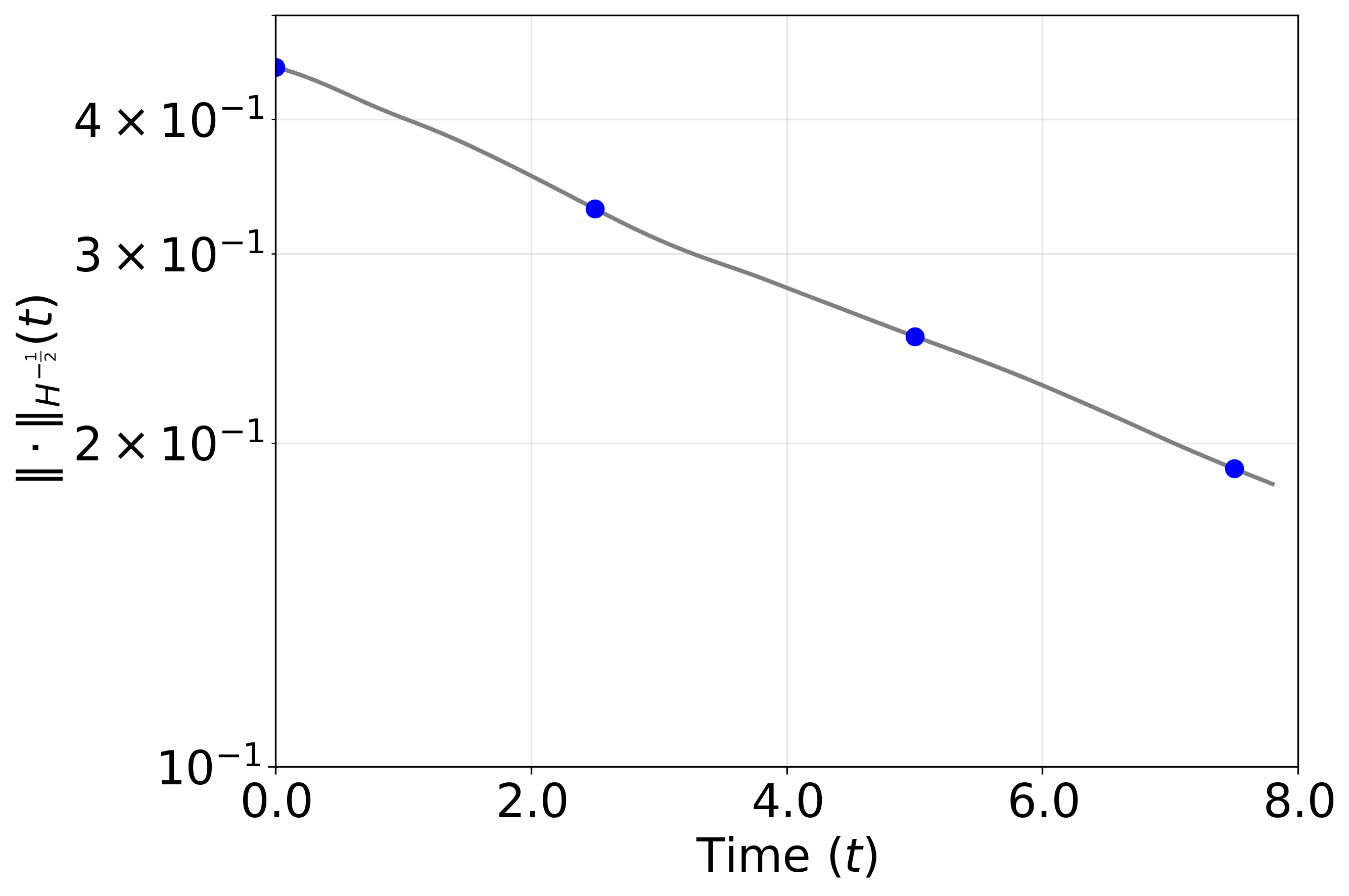}
        \caption{}
        \label{fig:norm half}
    \end{subfigure}
    \caption{shows mix norm $\norm{u(t, \cdot)}_{H^{-1}}$ and $\norm{u(t, \cdot)}_{H^{-\frac{1}{2}}}$ as a function of time. The resolution for this computation is $562 \times 562$.}
    \label{fig:norms vs time}
\end{figure}

\section{Discussion}
\label{sec: Discussion}
In this paper, we investigate the mixing of a divergence-free passive vector field \(u\) advected by another divergence-free vector field \(U\), where the evolution of \(u\) is governed by \eqref{eq:vectortransport}. As discussed in the introduction, this model serves as a testing ground for mechanisms of small-scale creation in divergence-free vector fields. The model \eqref{eq:vectortransport} is more tractable than the incompressible Euler equations, but at the same time substantially more challenging than the scalar transport equation due to the presence of pressure, which destroys the simple Lagrangian point of view available in the scalar setting. In this sense, model \eqref{eq:vectortransport} provides a useful framework for understanding the precise role of pressure in incompressible flows.

In \Cref{lem:existence} and \Cref{lem:stability}, we establish conditions on the vector field \(U\) that guarantee existence and uniqueness of solutions \(u\) to \eqref{eq:vectortransport}. We then derive lower bounds on mixing in \Cref{lem:case1}, \Cref{lem:case2} and \Cref{lem:case3} under a prescribed \(W^{1,q}\) constraint on \(U\). With the exception of the case \(q=\infty\), where we obtain an exponential lower bound, for finite values of \(q\) we are only able to prove a lower bound that allows mixing in finite time \(t_{\min}\). However, the time \(t_{\min}\) we obtain depends on the parameters \(q\), \(\alpha\) and \(d\). This finite-time lower bound stands in sharp contrast with the exponential lower bounds known for the scalar problem. We emphasize that the existing approaches that lead to exponential lower bounds in the scalar setting, for example, \cite{crippa2008regularity, seis2013maximal}, rely in an essential way on the Lagrangian viewpoint. Whether the finite-time lower bound obtained here can be improved to an exponential one remains an interesting open question.

We also carry out numerical simulations of mixing by choosing a field \(U\) that instantaneously maximizes the decay of the \({H}^{-\alpha}\)-norm of \(u\), in the spirit of the approach of Lin, Thiffeault and Doering \cite{Lin2011}. For both \(\alpha=\tfrac12\) and \(\alpha=1\), the numerical results indicate exponential mixing, suggesting that the optimal mixing rate is at least exponentially fast. One of the most striking features of the vector problem, which is in sharp contrast with passive scalar mixing, is that the passive vector field \(u\) not only develops small scales, but also concentrates on progressively smaller regions of the domain as time evolves.

There are several natural directions for future work. From the numerical point of view, one immediate next step is to extend the simulations to three dimensions. While one may still expect exponential mixing, it would be very interesting to find out the flow structures that emerge in \(u\) and \(U\), which are likely to differ from those observed in two dimensions. From the analysis point of view, the simulations suggest that \(u\) concentrates in thin regions, while the pressure becomes localized near regions where \(u\) bends or turns around. It would therefore be very interesting to investigate whether this numerically observed knowledge can be incorporated into a construction to provide a rigorous exponentially decreasing (in time) upper bound on the optimal mixing rate.

Finally, as promised in the introduction, we adopt a broader perspective and discuss how two conjectural questions for the Navier--Stokes equations can be posed at the level of the diffusive analogue of model \eqref{eq:vectortransport}, obtained by adding the term \(\nu \Delta u\) to the right-hand side. We believe that this model offers a simpler setting in which such questions may first be explored. Indeed, if they cannot be understood even in this simpler framework, then one should expect the corresponding questions for the Navier--Stokes equations to be considerably more difficult.

The first question is related to the fundamental phenomenon of \emph{anomalous dissipation} in turbulent flows . Informally, anomalous dissipation refers to the continual dissipation of energy in the vanishing viscosity limit \cite{sreenivasan1998update, frisch1995turbulence}. For the incompressible Navier--Stokes equations, this may be expressed as
\[
\varepsilon_\nu
=
\nu \langle |\nabla u_\nu|^2 \rangle
\geq c >0
\]
for all \(\nu \in (0,\nu_0]\), for some \(\nu_0>0\), where \(\langle \cdot \rangle\) denotes a suitable long-time volume average. Despite the abundance of observations supporting this phenomenon \cite{pearson2002measurements, kaneda2003energy}, to date there is still no rigorous example establishing it \cite{drivas2022self}.

Ideally, one would like to establish anomalous dissipation for a generic set of initial data. However, this appears to be an extremely difficult problem. In fact, proving anomalous dissipation even for a single initial condition is already highly nontrivial. The three main physical settings in which to study this question are:  
(i) a flow driven by a smooth large-scale forcing \(f \in C^\infty(\Omega)\), independent of \(\nu\),  
(ii) a boundary-driven flow, for example in a smooth domain, where energy is injected through the tangential velocity at the boundary,  
(iii) a freely decaying flow evolving from a smooth initial condition \(u_0\), independent of \(\nu\). In the third setting, the long-time average is naturally replaced by an average over a finite time interval.

In the context of the diffusive analogue of \eqref{eq:vectortransport}, a natural question is whether one can construct a family \(\{U_\nu\}\) exhibiting anomalous dissipation such that the corresponding solutions \(\{u_\nu\}\) exhibit the same property, in one of the three physical settings described above. A version of this question was recently considered in \cite{kumar2024almost}, where a weaker result was obtained. A family \(\{U_\nu\}\) was constructed for which the rate of energy dissipation tends to zero as \(\nu \to 0\), but only at a slow logarithmic rate, such that the same logarithmic decay was shown to be exhibited by the corresponding family \(\{u_\nu\}\).

Another question, we have in mind is related to singularity formation in the Navier–Stokes equation. For a smooth initial data, a classical conditional regularity result states that if the velocity $u$ belongs to the Prodi–Serrin class, i.e. $\frac{2}{p} + \frac{3}{q} = 1$, $q > 3$ (the case $q=3$ was covered by Escauriaza, Seregin, \v{S}ver\'ak \cite{escauriaza20023}), then $u$ is in fact a strong solution. Consequently, for blow-up to occur, the velocity must exit the Prodi–Serrin class. Motivated by this, we ask an analogous question at the level of the vector advection–diffusion equation. Suppose $U$ and $u$ satisfies the same smooth initial condition $u_0$, can one construct a velocity field $U$ that is smooth for all time $t < T$ and leaves the Prodi–Serrin class precisely at time $t = T$, in such a way that the corresponding solution $u$ also ceases to belong to the Prodi–Serrin class exactly at $t = T$?

\appendix
\section{Fractional Sobolev inequalities}
We recall the fractional Sobolev inequality from~\cite[Thm. 6.5]{DiNezza2012}:
\begin{theorem}[Fractional Sobolev inequality \cite{DiNezza2012}]
\label{thm:fractionalsobolev}
Let $s\in (0,1)$, $p\in [1,\infty)$ such that $sp<d$. Then there exists a positive constant $C=C(d,p,s)$ such that for any measurable function $f:\T^d\to \R$  we have
\begin{equation*}
	\norm{f}_{L^{p^*}}\leq C_S |f|_{W^{s,p}},
\end{equation*}
where the semi-norm on the right hand side is defined by 
\begin{equation*}
	|f|_{W^{s,p}}^p = \int_{\dom}\int_{\dom} \frac{|f(x)-f(y)|^p}{|x-y|^{d+sp}} dx dy,
\end{equation*}
and $p^*=p^*(d,s)$ is the `fractional critical exponent' and equal to $dp/(d-sp)$.
Consequently, the space $W^{s,p}(\dom)$ is continuously embedded in $L^q(\dom)$ for any $q\in [1,p^*]$. 
\end{theorem}

\begin{theorem}[Fractional Sobolev interpolation inequality~\cite{Brezis2018}]
	\label{thm:fractionalinterpolation}
	Let $\Omega \subset \mathbb {R} ^{d}$ be either the whole space, a half-space or a bounded Lipschitz domain. Let $1\leq p,\,p_{1},\,p_{2}\leq \infty$ be three positive extended real quantities and let $s,\,s_{1},\,s_{2}$ be non-negative real numbers. Furthermore, let $\theta \in (0,1)$ and assume that 
	\begin{equation*}
		s_{1}\leq s_{2},\qquad s=\theta s_{1}+(1-\theta )s_{2},\qquad {\dfrac {1}{p}}={\dfrac {\theta }{p_{1}}}+{\dfrac {1-\theta }{p_{2}}}
	\end{equation*} 
	hold. Then, 
	\begin{equation}\label{eq:fractionalGN}
		\|f\|_{W^{s,p}(\Omega )}\leq C_{fS}\|f\|_{W^{s_{1},p_{1}}(\Omega )}^{\theta }\|f\|_{W^{s_{2},p_{2}}(\Omega )}^{1-\theta }
	\end{equation} 
	for any $f\in W^{s_{1},p_{1}}(\Omega )\cap W^{s_{2},p_{2}}(\Omega )$  if and only if
	$$
	{\text{at least one of the following condition}}\quad {\begin{cases}s_{2}\in \mathbb {N} {\text{ and }}s_{2}\geq 1,\\p_{2}=1,\\0<s_{2}-s_{1}\leq 1-\displaystyle {\dfrac {1}{p_{1}}}\end{cases}}\quad {\text{is false.}}
	$$
	The constant $C_{fS}>0$ depends on the parameters $ p,\,p_{1},\,p_{2},\,s,\,s_{1},\,s_{2},\,\theta $, on the domain $ \Omega$, but not on $f$. 
\end{theorem}
\begin{theorem}[Fractional Gagliardo-Nirenberg-Sobolev inequality \cite{Brezis2019}]
	\label{thm:fractionalGNS}
	Let $\Omega\subset \R^d$ be the whole space, a half-space or a bounded Lipschitz domain.  Let $s_1, s_2, \tilde{r}, p_1, p_2,\tilde{q},\theta, d$ satisfy
	\begin{equation}\label{eq:conditionsGNS}
		\begin{split}
			&0\leq s_1\leq s_2, \, \tilde{r}\geq 0,\, 1\leq p_1,p_2,\tilde{q}\leq \infty,\, (s_1,p_1)\neq (s_2,p_2),\, \theta\in (0,1),\\
			& \tilde{r}<s:= \theta s_1 +(1-\theta) s_2,\, \frac{1}{\tilde{q}}=\left(\frac{\theta}{p_1}+\frac{1-\theta}{p_2}\right)-\frac{s-\tilde{r}}{d}.
		\end{split}
	\end{equation}
	Then the fractional Gagliardo-Nirenberg-Sobolev inequality 
	\begin{equation}
		\label{eq:fractionalGNS}
		\norm{f}_{W^{\tilde{r},\tilde{q}}}\leq C_{GNS} \norm{f}_{W^{s_1,p_1}}^\theta \norm{f}^{1-\theta}_{W^{s_2,p_2}}
	\end{equation}
	holds with the following exceptions, when it fails. 
	\begin{enumerate}
		\item $d=1$, $s_2\geq 1$ is an integer, $1<p_1\leq \infty$, $p_2=1$, $s_1=s_2-1+\frac{1}{p_1}$, and either $[1<p_1<\infty,\, \tilde{r}=s_2-1]$ or $\left[s_2+\frac{\theta}{p_1}-1<\tilde{r}<s_2+\frac{\theta}{p_1}-\theta\right]$.
		\item $N\geq 1$, $s_1<s_2$, $s_1-\frac{N}{p_1} = s_2-\frac{N}{p_2} =\tilde{r}$ is an integer, $\tilde{q}=\infty$, $(p_1,p_2)\neq (\infty,1)$ (for every $\theta\in (0,1)$).
	\end{enumerate}
	
\end{theorem}
\begin{theorem}[Fractional version of Morrey's inequality \cite{DiNezza2012}]
	\label{thm:fractionalmorrey}
	Let $\Omega\subset \R^d$ a Lipschitz domain and $p\in [1,\infty)$, $s\in (0,1)$ be such that $sp>d$. Then there exists $C>0$ depending on $d, s, p$ and $\Omega$ such that
	\begin{equation*}
		\norm{f}_{C^{0,\alpha}(\Omega)}\leq C \norm{f}_{W^{s,p(\Omega)}},
	\end{equation*}
	for any $f\in L^p$ and $\alpha:= (sp-d)/p$.
\end{theorem}
\bibliographystyle{abbrv}
\bibliography{mixing}

@Article {AlbertiCrippaMazzucato19,
    AUTHOR = {Alberti, Giovanni and Crippa, Gianluca and Mazzucato, Anna L.},
     TITLE = {Exponential self-similar mixing by incompressible flows},
   JOURNAL = {J. Amer. Math. Soc.},
  FJOURNAL = {Journal of the American Mathematical Society},
    VOLUME = {32},
      YEAR = {2019},
    NUMBER = {2},
     PAGES = {445--490},
      ISSN = {0894-0347},
   MRCLASS = {35Q35 (76F25)},
  MRNUMBER = {3904158},
MRREVIEWER = {Maurizio Brocchini},
       DOI = {10.1090/jams/913},
       URL = {https://doi.org/10.1090/jams/913},
}

@Article{bressan2003lemma,
  title={A lemma and a conjecture on the cost of rearrangements},
  author={Bressan, Alberto},
  journal={Rendiconti del Seminario Matematico della Universita di Padova},
  volume={110},
  pages={97--102},
  year={2003}
}

@Article {Brezis2018,
	AUTHOR = {Brezis, Ha\"im and Mironescu, Petru},
	TITLE = {Gagliardo-{N}irenberg inequalities and non-inequalities: the
	full story},
	JOURNAL = {Ann. Inst. H. Poincar\'e{} C Anal. Non Lin\'eaire},
	FJOURNAL = {Annales de l'Institut Henri Poincar\'e{} C. Analyse Non
	Lin\'eaire},
	VOLUME = {35},
	YEAR = {2018},
	NUMBER = {5},
	PAGES = {1355--1376},
	ISSN = {0294-1449,1873-1430},
	MRCLASS = {35A23 (26D10)},
	MRNUMBER = {3813967},
	MRREVIEWER = {Jean\ Van Schaftingen},
	DOI = {10.1016/j.anihpc.2017.11.007},
	URL = {https://doi.org/10.1016/j.anihpc.2017.11.007},
}

@Article {Brezis2019,
	AUTHOR = {Brezis, Ha\"im and Mironescu, Petru},
	TITLE = {Where {S}obolev interacts with {G}agliardo-{N}irenberg},
	JOURNAL = {J. Funct. Anal.},
	FJOURNAL = {Journal of Functional Analysis},
	VOLUME = {277},
	YEAR = {2019},
	NUMBER = {8},
	PAGES = {2839--2864},
	ISSN = {0022-1236,1096-0783},
	MRCLASS = {46E35},
	MRNUMBER = {3990737},
	MRREVIEWER = {Jes\'us\ A.\ Jaramillo},
	DOI = {10.1016/j.jfa.2019.02.019},
	URL = {https://doi.org/10.1016/j.jfa.2019.02.019},
}

@Article{brue2024sharp,
  title={{Sharp Nonuniqueness in the Transport Equation with Sobolev Velocity Field}},
  author={Bru{\`e}, Elia and Colombo, Maria and Kumar, Anuj},
  journal={arXiv preprint arXiv:2405.01670},
  year={2024}
}

@Article{caldwell1995turbulence,
  title={Turbulence and mixing in the ocean},
  author={D. R. Caldwell and J. N. Mourn},
  journal={Rev. Geophys. },
  volume={33},
  number={S2},
  pages={1385--1394},
  year={1995},
  publisher={Wiley Online Library}
}

@Article{crippa2008regularity,
  title={Regularity and compactness for the DiPerna--Lions flow},
  author={Crippa, Gianluca and De Lellis, Camillo},
  journal={Hyperbolic problems: theory, numerics, applications},
  pages={423--430},
  year={2008},
  publisher={Springer}
}

@Article{depauw2003non,
  title={{Non unicit{\'e} des solutions born{\'e}es pour un champ de vecteurs BV en dehors d'un hyperplan}},
  author={Depauw, Nicolas},
  journal={Comptes rendus. Math{\'e}matique},
  volume={337},
  number={4},
  pages={249--252},
  year={2003},
  doi = {10.1016/S1631-073X(03)00330-3}
}

@Article {DiNezza2012,
	AUTHOR = {Di Nezza, Eleonora and Palatucci, Giampiero and Valdinoci,
	Enrico},
	TITLE = {Hitchhiker's guide to the fractional {S}obolev spaces},
	JOURNAL = {Bull. Sci. Math.},
	FJOURNAL = {Bulletin des Sciences Math\'ematiques},
	VOLUME = {136},
	YEAR = {2012},
	NUMBER = {5},
	PAGES = {521--573},
	ISSN = {0007-4497,1952-4773},
	MRCLASS = {46E35 (35A23 35S05 35S30)},
	MRNUMBER = {2944369},
	MRREVIEWER = {Lanzhe\ Liu},
	DOI = {10.1016/j.bulsci.2011.12.004},
	URL = {https://doi.org/10.1016/j.bulsci.2011.12.004},
}

@Article{doering2019optimal,
  title={On the Optimal Design of Wall-to-Wall Heat Transport},
  author={Doering, Charles R and Tobasco, Ian},
  journal={Comm. Pure Appl. Math.},
  volume={72},
  number={11},
  pages={2385--2448},
  year={2019},
  publisher={Wiley Online Library},
  DOI = {10.1002/cpa.21832}
}

@Article {drivas22anomdissp,
    AUTHOR = {Drivas, Theodore D. and Elgindi, Tarek M. and Iyer, Gautam and
              Jeong, In-Jee},
     TITLE = {Anomalous dissipation in passive scalar transport},
   JOURNAL = {Arch. Ration. Mech. Anal.},
  FJOURNAL = {Archive for Rational Mechanics and Analysis},
    VOLUME = {243},
      YEAR = {2022},
    NUMBER = {3},
     PAGES = {1151--1180},
      ISSN = {0003-9527},
   MRCLASS = {76F02 (76R50)},
  MRNUMBER = {4381138},
       DOI = {10.1007/s00205-021-01736-2},
       URL = {https://doi.org/10.1007/s00205-021-01736-2},
}

@Article{drivas2022self,
  title={{Self-regularization in turbulence from the Kolmogorov 4/5-law and alignment}},
  author={Drivas, Theodore D},
  journal={Philosophical Transactions of the Royal Society A},
  volume={380},
  number={2226},
  pages={20210033},
  year={2022},
  publisher={The Royal Society}
}

@Article {elgindiuniversalmixer,
    AUTHOR = {Elgindi, Tarek M. and Zlato\v{s}, Andrej},
     TITLE = {Universal mixers in all dimensions},
   JOURNAL = {Adv. Math.},
  FJOURNAL = {Advances in Mathematics},
    VOLUME = {356},
      YEAR = {2019},
     PAGES = {106807, 33},
      ISSN = {0001-8708},
   MRCLASS = {35Q30 (35F05 37B20)},
  MRNUMBER = {4008523},
MRREVIEWER = {Beno\^{\i}t P. Desjardins},
       DOI = {10.1016/j.aim.2019.106807},
       URL = {https://doi.org/10.1016/j.aim.2019.106807},
}

@article{elgindi2023norm,
  title={Norm growth, non-uniqueness, and anomalous dissipation in passive scalars},
  author={Elgindi, Tarek M and Liss, Kyle},
  journal={Archive for Rational Mechanics and Analysis},
  volume={248},
  number={6},
  pages={120},
  year={2024},
  publisher={Springer}
}

@article{escauriaza20023,
  title={On L 3, Infinity-solutions to the Navier-Stokes equations and backward uniqueness},
  author={Escauriaza, Luis and Seregin, Gregory and Sverak, Vladimir},
  year={2002}
}

@Article{iyer2014lower,
  title={Lower bounds on the mix norm of passive scalars advected by incompressible enstrophy-constrained flows},
  author={Iyer, Gautam and Kiselev, Alexander and Xu, Xiaoqian},
  journal={Nonlinearity},
  volume={27},
  number={5},
  pages={973},
  year={2014},
  publisher={IOP Publishing}
}

@Article{kaneda2003energy,
  title={Energy dissipation rate and energy spectrum in high resolution direct numerical simulations of turbulence in a periodic box},
  author={Kaneda, Yukio and Ishihara, Takashi and Yokokawa, Mitsuo and Itakura, Ken’ichi and Uno, Atsuya},
  journal={Physics of Fluids},
  volume={15},
  number={2},
  pages={L21--L24},
  year={2003},
  publisher={American Institute of Physics}
}

@article{kumar2024almost,
  title={Almost anomalous dissipation in advection-diffusion of a divergence-free passive vector},
  author={Kumar, Anuj},
  journal={arXiv preprint arXiv:2409.15000},
  year={2024}
}

@article{kumar2022three,
  title={Three dimensional branching pipe flows for optimal scalar transport between walls},
  author={Kumar, Anuj},
  journal={Nonlinearity},
  volume={37},
  number={11},
  pages={115011},
  year={2024},
  publisher={IOP Publishing}
}

@article{kumar2023nonuniqueness,
  title={Nonuniqueness of trajectories on a set of full measure for Sobolev vector fields},
  author={Kumar, Anuj},
  journal={Archive for Rational Mechanics and Analysis},
  volume={248},
  number={6},
  pages={114},
  year={2024},
  publisher={Springer}
}

@Article {Lin2011,
    AUTHOR = {Lin, Zhi and Thiffeault, Jean-Luc and Doering, Charles R.},
     TITLE = {Optimal stirring strategies for passive scalar mixing},
   JOURNAL = {J. Fluid Mech.},
  FJOURNAL = {Journal of Fluid Mechanics},
    VOLUME = {675},
      YEAR = {2011},
     PAGES = {465--476},
      ISSN = {0022-1120,1469-7645},
   MRCLASS = {76D55 (35K60 49N90 76W05)},
  MRNUMBER = {2801050},
MRREVIEWER = {Jean-Pierre\ Raymond},
       DOI = {10.1017/S0022112011000292},
       URL = {https://doi.org/10.1017/S0022112011000292},
}

@Article{munk1998abyssal,
  title={Abyssal recipes {II}: {Energetics} of tidal and wind mixing},
  author={W. Munk and C. Wunsch},
  journal={Deep Sea Research Part I: Oceanographic Research Papers},
  volume={45},
  number={12},
  pages={1977--2010},
  year={1998},
  publisher={Elsevier}
}

@article{ottino1994mixing,
  title={Mixing and chemical reactions a tutorial},
  author={Ottino, Julio M},
  journal={Chemical Engineering Science},
  volume={49},
  number={24},
  pages={4005--4027},
  year={1994},
  publisher={Elsevier}
}

@Article{pearson2002measurements,
  title={Measurements of the turbulent energy dissipation rate},
  author={Pearson, Bruce R and Krogstad, P-{\AA} and van de Water, Willem},
  journal={Physics of fluids},
  volume={14},
  number={3},
  pages={1288--1290},
  year={2002},
  publisher={American Institute of Physics}
}

@article{sreenivasan1984scaling,
  title={On the scaling of the turbulence energy dissipation rate},
  author={K. R. Sreenivasan},
  journal={Phys. Fluids},
  volume={27},
  number={5},
  pages={1048--1051},
  year={1984},
  publisher={American Institute of Physics}
}

@article{sreenivasan1998update,
  title={An update on the energy dissipation rate in isotropic turbulence},
  author={Sreenivasan, Katepalli R},
  journal={Physics of Fluids},
  volume={10},
  number={2},
  pages={528--529},
  year={1998},
  publisher={American Institute of Physics}
}

@Article{seis2013maximal,
  title={Maximal mixing by incompressible fluid flows},
  author={Seis, Christian},
  journal={Nonlinearity},
  volume={26},
  number={12},
  pages={3279},
  year={2013},
  publisher={IOP Publishing}
}

@Article{thiffeault2012using,
  title={Using multiscale norms to quantify mixing and transport},
  author={Thiffeault, Jean-Luc},
  journal={Nonlinearity},
  volume={25},
  number={2},
  pages={R1--R44},
  year={2012}
}

@Article{tobasco2021optimal,
  title={Optimal cooling of an internally heated disc},
  author={Tobasco, Ian},
  journal={Philosophical Transactions of the Royal Society A: Mathematical, Physical and Engineering Sciences},
  volume={380},
  number={2225},
  year={2022},
  publisher={The Royal Society}
}

@Article{wunsch2004vertical,
  title={Vertical mixing, energy, and the general circulation of the oceans},
  author={C. Wunsch and R. Ferrari},
  journal={Annu. Rev. Fluid Mech.},
  volume={36},
  pages={281--314},
  year={2004},
  publisher={Annual Reviews}
}

@Article {YaoZlatos17,
    AUTHOR = {Yao, Yao and Zlato\v{s}, Andrej},
     TITLE = {Mixing and un-mixing by incompressible flows},
   JOURNAL = {J. Eur. Math. Soc. (JEMS)},
  FJOURNAL = {Journal of the European Mathematical Society (JEMS)},
    VOLUME = {19},
      YEAR = {2017},
    NUMBER = {7},
     PAGES = {1911--1948},
      ISSN = {1435-9855},
   MRCLASS = {76F25 (35B45 35L65 35Q35)},
  MRNUMBER = {3656475},
       DOI = {10.4171/JEMS/709},
       URL = {https://doi.org/10.4171/JEMS/709},
}

@article {Simon1987,
    AUTHOR = {Simon, Jacques},
     TITLE = {Compact sets in the space {$L^p(0,T;B)$}},
   JOURNAL = {Ann. Mat. Pura Appl. (4)},
  FJOURNAL = {Annali di Matematica Pura ed Applicata. Serie Quarta},
    VOLUME = {146},
      YEAR = {1987},
     PAGES = {65--96},
      ISSN = {0003-4622},
   MRCLASS = {46E40 (46E30)},
  MRNUMBER = {916688},
MRREVIEWER = {James\ Bell\ Cooper},
       DOI = {10.1007/BF01762360},
       URL = {https://doi.org/10.1007/BF01762360},
}

@book{frisch1995turbulence,
  title={{Turbulence: the legacy of AN Kolmogorov}},
  author={Frisch, Uriel},
  year={1995},
  publisher={Cambridge university press}
}

@book{paul2004handbook,
  title={Handbook of industrial mixing},
  author={Paul, Edward L and Atiemo-Obeng, Victor A and Kresta, Suzanne M and others},
  year={2004},
  publisher={Wiley Online Library}
}

\end{document}